\documentclass[11pt,a4paper,reqno]{amsart}
\usepackage[hmargin=2cm,bmargin=3.2cm,tmargin=3.2cm]{geometry}
\usepackage{enumerate}
\usepackage{amsmath,amsthm,amssymb,amsfonts,latexsym}
\usepackage{mathbbol}

\usepackage{esvect}

\usepackage[colorlinks, linkcolor=blue, filecolor=blue,
 citecolor = olive, urlcolor=purple]{hyperref}
\usepackage[gen]{eurosym}
\usepackage[english]{babel}
\usepackage[normalem]{ulem}
\usepackage{latexsym}
\usepackage{orcidlink}
\usepackage{tikz}					
\usetikzlibrary{arrows.meta, decorations.markings,calc}

\DeclareMathOperator{\dd}{\mbox{d}\!} 

\DeclareMathOperator{\Ker}{Ker}

\DeclareMathOperator{\FORM}{FORM}

\DeclareMathOperator{\Id}{Id}

\newcommand{\e}{\mathrm{e}}

 \def\mG{\mathsf{G}}
 \def\mH{\mathsf{H}}
 \def\mV{\mathsf{V}}
 \def\mE{\mathsf{E}}
 
 \def\mK{\mathsf{K}}

 \def\mK{\mathsf{K}}

 \def\mv{\mathsf{v}}
 \def\mw{\mathsf{w}}
 \def\mz{\mathsf{z}}
 \def\me{\mathsf{e}}
 \def\mesour{{\mathsf{e}_{\rm sour}}}
 \def\metarg{{\mathsf{e}_{\rm targ}}}
 \def\mfsour{{\mathsf{f}_{\rm sour}}}
 \def\mftarg{{\mathsf{f}_{\rm targ}}}
 \def\mw{\mathsf{w}}
 \def\mz{\mathsf{z}}
 \def\mf{\mathsf{f}}

\newcommand{\K}{\mathbb{K}}
\newcommand{\R}{\mathbb{R}}
\newcommand{\N}{\mathbb{N}}

\newcommand{\C}{\mathbb{C}}

\def\XXint#1#2#3{{\setbox0=\hbox{$#1{#2#3}{\int}$ }
\vcenter{\hbox{$#2#3$ }}\kern-.6\wd0}}

\usepackage{aliascnt}

\theoremstyle{plain}

\newaliascnt{cor}{theo}
\newaliascnt{prop}{theo}
\newaliascnt{lemma}{theo}

\newtheorem{lemma}[lemma]{Lemma}
\newtheorem{prop}[prop]{Proposition}
\newtheorem{cor}[cor]{Corollary}

\aliascntresetthe{cor}
\aliascntresetthe{prop}
\aliascntresetthe{lemma}

\theoremstyle{defi} 

\newaliascnt{defi}{theo}
\newaliascnt{assum}{theo}
\newaliascnt{assums}{theo}
\newaliascnt{prob}{theo}

\newtheorem{defi}[defi]{Definition}

\aliascntresetthe{defi}
\aliascntresetthe{assum}
\aliascntresetthe{assums}
\aliascntresetthe{prob}

\theoremstyle{rem}

\newaliascnt{rems}{theo}
\newaliascnt{rem}{theo}
\newaliascnt{exa}{theo}
\newaliascnt{exs}{theo}

\newtheorem{rem}[rem]{Remark}
\newtheorem{exa}[exa]{Example}

\aliascntresetthe{rems}
\aliascntresetthe{rem}
\aliascntresetthe{exa}
\aliascntresetthe{exs}

\numberwithin{equation}{section}
\numberwithin{lemma}{section}

\newcommand{\be}{\begin{equation}}
\newcommand{\ee}{\end{equation}}

\newcommand{\beq}{\begin{eqnarray}}
\newcommand{\eeq}{\end{eqnarray}}

\title{Non-Markovian heat flows on directed hypergraphs} 

\subjclass[2010]{}

\keywords{}

\author[D.~Mugnolo]{Delio Mugnolo
\orcidlink{0000-0001-9405-0874}}

\address{Lehrgebiet Analysis, Fakult\"at Mathematik und Informatik, Fern\-Universit\"at in Hagen, D-58084 Hagen, Germany}
\email{delio.mugnolo@fernuni-hagen.de}

\date{\today}

\thanks{
The author was partially supported by the Deutsche Forschungsgemeinschaft DFG (Grant 397230547)}
\keywords{Hypergraphs; Dirichlet forms; Eventual positivity}
\subjclass[2010]{47D06, 05C50, 05C65, 39A12}

\begin{document}

\maketitle
\begin{abstract}

We introduce a semigroup  framework for Laplacians on directed hypergraphs, extending the classical heat flow models on graphs and establishing hypergraphs as prototypical models for non-Markovian diffusion. We  apply spectral surgery methods to
derive eigenvalue bounds, thus describing large-time behaviour of the heat flow. Unlike on standard graphs, heat flows on directed hypergraphs may lose positivity and/or $\infty$-contractivity, yet can recover them eventually or asymptotically under specific combinatorial configurations:
examples based on duals of oriented graph and  realisations of the Fano plane illustrate these phenomena.
 Our approach combines  combinatorial, order-theoretic and linear-algebraic methods.

.\end{abstract}

\section{Introduction}
The analysis of diffusion processes on discrete structures has long revolved around the Laplacian on graphs. In the classical setting, the interplay between the generator’s order structure and the geometry of the underlying space is fully captured by the Beurling-Deny criteria: self-adjoint operators that generate positive, or even Markov, semigroups correspond to quadratic forms enjoying suitable lattice and contractivity properties. These results provide the functional-analytic backbone of the theory of Markov processes and of the modern calculus of Dirichlet forms \cite{Ouh05,FukOshTak11,BakGenLed14} and their implications on graphs are now widely understood \cite{KelLenWoj21}: roughly speaking, each sub-Markovian semigroup on a discrete space is generated by a Schrödinger operator on a graph.

Extending this theory to directed hypergraphs -- where edges connect \textit{sets} of sources to \textit{sets} of targets -- reveals fundamentally new behaviour: the corresponding heat flow may lose positivity or mass preservation.
Understanding how such failures occur is crucial if one wishes to view hypergraph Laplacians not merely as combinatorial curiosities, but as models of higher-order dynamics, much in the spirit of \cite{Bia21b}.
In this article we develop a semigroup framework for such Laplacians, identifying general conditions under which unusual diffusive behaviours vanish asymptotically or even after a finite transient. Therefore, we regard this article as a contribution to the field of non-Markovian diffusion \cite{Mur08}, which has recently become rather popular. We will show that on directed hypergraphs the Markovian property can break down in many different ways: they are thus a very convenient environment to test different anomalous diffusive behaviours.

Hypergraphs, first introduced by Berge \cite{Ber73}, are a now classical topic of discrete mathematics: they naturally generalise graphs to incidence structures in which more than two points are allowed to belong to a so-called \textit{hyperedge}. Beyond their rich mathematical theory, hypergraphs have seen many natural applications in different fields of applied sciences, including operations research and transportation networks \cite{GalLonPal93}, molecular modelling \cite{KonSko95}, neural networks \cite{FenYouZha19} and much more, see~\cite{Bre13} for an overview.
 
In this article we introduce a hypergraph Laplacian -- and, in turn, a heat flow -- following an approach that, to the best of our knowledge, was first proposed in \cite{RefRus12} and, in a normalised version, in~\cite{SaiManSuz18,JosMul19}, following the blueprint of directed graphs, upon considering a hypergraph version of the divergence and gradient operators. The crucial difference from earlier \textit{Ansätze} that will be overviewed below is that hypergraphs \textit{must} here be directed (i.e., each hyperedge is assigned an orientation by fixing a distinct set of source and target endpoints). The importance of this paradigm for cellular networks and non-equilibrium thermodynamics has been observed in~\cite{KlaHauThe09, DalLecPol23} respectively: in particular, in \cite{PolEsp14}, motivated by relevant chemical models, the authors replace incidence matrices -- that is, discrete divergence operators -- by more general \textit{stoichiometric matrices} which, in turn, also induce a relaxed notion of Laplacian.

Indeed, the motivation in much of the work on hypergraphs in the chemical-physical community seems to be of homological nature: the goal is to understand the nature of cycles and co-cycles of such stoichiometric matrices and relate them to the original chemical reactions. On the other hand, several papers in applied linear algebra  have focused on a hypergraph approach to image processing and  data analysis, see \cite{SaiManSuz18,JosMul19}.
 Our goal is different: Inspired by the theory of discrete Dirichlet spaces, the main aim of this article is to expand on the note~\cite{Mug14d} and discuss the properties of the linear dynamical system associated with such hypergraph Laplacians. 
 
 In view of the linearity of our finite dimensional model, it is clear that such a dynamical system is globally (both forward and backward) well-posed: its implications in modelling have been discussed already in~\cite{MulKueJos20,JosMulZha22,Faz23,FazTenLuk24}, also in the context of nonlinear heat flows. We are going to deepen this investigation by studying relevant qualitative properties of this flow. In particular, we identify a few combinatorial cases in which the evolution equation driven by such Laplacian is governed by a (sub-)Markovian semigroup: this is always the case on graphs, but we will see that on hypergraphs the sub-Markov property is a rare and precious property. This makes heat flows on hypergraphs to a prototypical field of application of the young, fascinating area of eventually positive semigroup \cite{Glu22}: Many  instances of non-positive semigroups have been found in the last years, those generated by the Dirichlet-to-Neumann operators, by Laplacians, with nonlocal boundary conditions, or by powers of Laplacians, both on domains and graphs \cite{DanGluKen16,GreMug20,BecGreMug21,DenKunPlo21,
 GluMui24}. We further elaborate on these ideas by introducing the notion of asymptotically $\infty$-contractive semigroup, which we fully characterise in \autoref{prop:eventlinftycontr} and \autoref{cor:eventlinftycontr-matr}: this may be of independent interest.

\smallskip
Let us sketch the structure of the present paper.
In Section~\ref{sec:general} we present the relevant notion of Laplacian $\mathcal L$ on directed hypergraphs and encounter the associated heat flow.
For the purpose of studying the large-time behaviour of the semigroup $(\e^{-t\mathcal L})_{t\ge 0}$, it is important to find bounds on the lower eigenvalues of $\mathcal L$, especially its lower ones: this will be done in Section~\ref{sec:spectheor}, with a focus on how eigenvalues change if an elementary hypergraph is perturbed.

It is not difficult to find examples showing that the heat flow on directed hypergraphs may fail to be merely \textit{sub}-Markovian: they may lack the positivity or the stochasticity property (or both); but, in some cases, such heat flows turn out to be \textit{eventually} sub-Markovian. We discuss these properties in Section~\ref{sec:ordertheor}.

Directed hypergraphs are special instances of incidence structures; as such, it is natural to study their dual structures. The theory of directed hypergraphs that are duals of directed graphs turns out to be rich and worth exploring. We will do so in Section~\ref{sec:graphypergra}, also elaborating on the interplays with theory of simplicial complexes: we put forward the opinion that diffusion on directed hypergraphs is a convenient relaxation of diffusion on  simplicial complexes.

We conclude this article by discussing in Section~\ref{sec:fano} the case of an especially nice class of directed hypergraphs: the oriented realisations of the Fano plane, also commenting on the similarities and the differences of the heat flow driven by these realisations.

\subsection*{Related work: An overview of other approaches for analysis on hypergraphs}

Several approaches to discrete analysis on hypergraphs are available and, in particular, several notions of adjacency and Laplace structures on hypergraphs have been proposed in the literature: let us briefly review them, partially elaborating on the discussion in~\cite{MulKueBoe22}. The easiest one -- considered, e.g., in \cite{ZhoHuaSch07,ProBenTud22,SaiHer23,MazSolTol23} -- suggests to replace each hyperedge containing $n$ vertices with $\frac{n(n-1)}{2}$ edges pairwise connecting any two vertices, in accordance with Berge's definition of a hypergraph's \textit{2-section}; and to regard the Laplacian on such 2-section as a hypergraph Laplacian. This is a robust approach, also used for the study of centrality measures in several articles by Tudisco and his co-authors (see~\cite{TudHig21} and references therein), but we argue that, in this way, fundamental features of a higher order networks are lost. A comparison of the random walk on the  hypergraphs' 2-section and random walks induced by (adapted) versions of other  hypergraph Laplacians has been performed in~\cite{MulKueBoe22}.
A more sophisticated way of turning a hypergraph into a graph and then defining on it a nonlinear Laplacian (more precisely: the subdifferential of the total variation) was put forward in \cite{HeiSetJos13}, who also discussed the variational aspect of this theory. A different but related approach was explored in~\cite{Lou15} and later interpreted in~\cite{Yos19} in the framework of submodular transformations: using this approach, diffusion equations have been studied in~\cite{TakMiyIke20,IkeMiyTak22,IkeUch23}; in particular, the large-time behaviour of the (possibly nonlinear) heat flow is studied in \cite{IkeUch23}. Two different tensorial approaches to defining the adjacency matrix of directed and undirected hypergraphs were suggested in \cite{FriWid95} and in~\cite{CooDut12,PeaZha14,Pea15}: in these articles, the focus was on the properties on describing the properties of the two largest eigenvalues and the expander property defined in terms of the second largest eigenvalue of such adjacency matrix; and on the development of a general algebraic graph theory, respectively. A different, more combinatorial way of defining the adjacency matrix of an undirected hypergraph was pursued in~\cite{FenLi96}, again with the goal of providing bounds for the second eigenvalue and of studying the behaviour as the number of vertices tends to $\infty$.
Yet another different structure was suggested in \cite{Chu93}: it borrows ideas of the theory of simplicial complexes and is, therefore, restricted to hypergraphs displaying certain regularity properties: in particular, each hyperedge must contain the same number of vertices (the theory of Laplacian of simplicial complexes is, indeed, very well studied, see e.g.\ \cite{HorJos13b,ParRosTes16,Che20,BarKel25}). In~\cite{KurMulNab25}, a description of undirected hypergraphs inspired by the theory of metric graphs is proposed, without explicitly suggesting how to use it to introduce a Laplacian on directed hypergraphs.
Also, the Moore--Penrose inverse of a graph Laplacian has been recently interpreted as a hypergraph Laplacian in~\cite{Est25}.

\section{General properties of directed hypergraphs}\label{sec:general}

Directed hypergraphs were introduced independently by several authors: we refer to~\cite{GalLonPal93} for a historical overview. Let us recall the fundamental ideas.

\begin{defi}\label{def:hyper}
 A \emph{directed hypergraph} is a pair $\mH=(\mV,\mE)$, where $\mV$ is a set and $\mE$ is a family (more precisely, a multiset) of pairs, $\me=(\me_{\rm sour},\me_{\rm targ})$, such that $\me_{\rm sour},\me_{\rm targ}\in {\mathcal P}(\mV)$ with $\me_{\rm sour}\cap \me_{\rm targ}=\emptyset$: we sometimes also use the notation
\[
\begin{split}
\me&=
\vv{\{\mv_1,\ldots,\mv_h\},\{\mw_1,\ldots,\mw_k\}}\qquad \hbox{with $h,k\in \mathbb N_0$.}\\
&\qquad{=:\me_{\rm sour}}\qquad \quad{=:\me_{\rm targ}}
\end{split}
\]
The elements of $\mV$ and $\mE$ are called \emph{vertices} and \emph{hyperedges}, respectively.
We call $\me_{\rm sour}$ (resp., $\me_{\rm targ}$) the \emph{source} (resp., \emph{target}) \emph{endsets} of the hyperedge $\me$.
\end{defi}

Throughout this article we  restrict to the case of finite hypergraphs, i.e., $\mV$ and hence $\mE$ are finite. Observe that if $\me=(\me_{\rm sour},\me_{\rm targ})$ is a hyperedge in $\mH$, then the above definition does not prevent either the ``reversed'' hyperedge $(\me_{\rm sour},\me_{\rm targ})$ or another copy of $\me$ to possibly belong to $\mE$, too: this will be important in the following, for example in the construction in \autoref{def:union}. In this sense, \autoref{def:hyper} is the generalisation of the notion of directed multigraph to the hypergraph context.

Inspired by the  case of directed graphs, we introduce  a notion of Laplacian on directed hypergraphs.

\begin{defi}
Let $\mH$ be a directed hypergraph.
The \emph{(signed) incidence matrix} of $\mH$ is the $\mV\times \mE$ matrix $\mathcal I=(\iota_{\mv\me})$ given by
$${\iota}_{\mv \me}:=\left\{
\begin{array}{rl}
+1 & \hbox{if}~\mv\in \me_{\rm targ},\\
-1 & \hbox{if}~\mv\in \me_{\rm sour},\\
0 & \hbox{otherwise}.
\end{array}\right.$$ 
The associated \emph{Laplacian} is the $\mV\times\mV$-matrix 
\[
\mathcal L:=\mathcal I \mathcal I^\top .
\]
\end{defi}

Clearly, $\iota_{\mv\me}\iota_{\mw\me}\in \{-1,0,+1\}$ for all $\mv,\mw\in \mV$ and all $\me\in \mE$: 
more precisely,
\begin{itemize}
\item $\iota_{\mv\me}\iota_{\mw\me}=+1$ if and only if $\mv,\mw\in \me_{\rm sour}$; or $\mv,\mw\in \me_{\rm targ}$;
\item $\iota_{\mv\me}\iota_{\mw\me}=-1$ if and only if $\mv\in \me_{\rm sour}$ and $\mw\in \me_{\rm targ}$; or $\mv\in \me_{\rm targ}$ and $\mw\in \me_{\rm sour}$;
\item $\iota_{\mv\me}\iota_{\mw\me}=0$ if and only if $\me\not\in \mE_\mv\cap \mE_\mw$.
\end{itemize}
Following~\cite{MulHorJos22}, in the former (resp., latter) case we say that $\mv,\mw$ are \textit{co-oriented} (resp., \textit{anti-oriented}) in $\me$. 

We use throughout the following notation:
\begin{itemize}
\item $\mE_\mv$ is the set of all hyperedges which $\mv$ belongs to;
\item $\mE^{\backslash\!\backslash}_{\mv\mw}$ is the set of all hyperedges in which $\mv,\mw$ are co-oriented;
\item $\mE^{\not\backslash\!\backslash}_{\mv\mw}$ is the set of all hyperedges in which $\mv,\mw$ are anti-oriented.
\end{itemize}

We denote by $\#S$ the number of elements of a set $S$: so $\#\mV,\#\mE$ are the number of vertices and hyperedges in a hypergraph, respectively; $\#\mesour$ is the number of vertices in the source endset of a hyperedge $\me$, and so on. In particular, in analogy with the classical case of graphs, we hence use the following notations:
\[
\deg(\mv):=\#\mE_\mv,\quad \deg_{\min}(\mV):=\min_{\mv\in\mV}\deg(\mv),\quad \deg_{\max}(\mV):=\max_{\mv\in\mV}\deg(\mv).
\]

\begin{exa}
Consider two hypergraphs $\mH_1,\mH_2$, each consisting of three vertices and one hyperedge, as depicted in \autoref{fig:twobasiccg}.

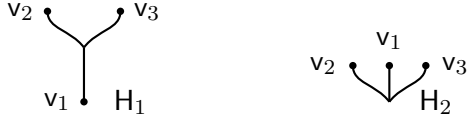
\begin{figure}[ht]
\begin{tikzpicture}[scale=.6, thick]

  \coordinate (v1) at (0,-.5);
  \coordinate (split) at (0,0.7);
  \coordinate (v2) at (-0.8,1.5);
  \coordinate (v3) at (0.8,1.5);

\draw (v1) -- (split);
  \draw[thick] 
    (split) .. controls +(-0.3,0.4) and +(0,-0.4) .. (v2);
  \draw[thick] 
    (split) .. controls +(0.3,0.4) and +(0,-0.4) .. (v3);


  \fill (v1) circle (2.3pt) node[left=1pt] {$\mv_1$};
  \fill (v2) circle (2.3pt) node[left=1pt] {$\mv_2$};
  \fill (v3) circle (2.3pt) node[right=1pt] {$\mv_3$};
  \fill (v1) circle (0pt) node[right=7pt] {$\mH_1$};
\end{tikzpicture}
\qquad\qquad
\begin{tikzpicture}[scale=.6, thick]

  \coordinate (v1) at (0,1.5);
  \coordinate (split) at (0,0.7);
  \coordinate (v2) at (-0.8,1.5);
  \coordinate (v3) at (0.8,1.5);

\draw (v1) -- (split);
  \draw[thick] 
    (split) .. controls +(-0.3,0.4) and +(0,-0.4) .. (v2);
  \draw[thick] 
    (split) .. controls +(0.3,0.4) and +(0,-0.4) .. (v3);


  \fill (v1) circle (2.3pt) node[above=2pt] {$\mv_1$};
  \fill (v2) circle (2.3pt) node[left=2pt] {$\mv_2$};
  \fill (v3) circle (2.3pt) node[right=2pt] {$\mv_3$};
  \fill (split) circle (0pt) node[right=7pt] {$\mH_2$};
  
\end{tikzpicture}\caption{Two hypergraphs consisting of three vertices and one hyperedge}\label{fig:twobasiccg}
\end{figure}

The corresponding incidence matrices are
\[
\mathcal I_1=\begin{pmatrix}
-1 \\ 1 \\ 1
\end{pmatrix},\qquad 
\mathcal I_2=\begin{pmatrix}
1 \\ 1 \\ 1
\end{pmatrix}.
\]

Then we find for $\mH_1$:
\[
\#\mE_{\mv_1}=\#\mE_{\mv_3}=\#\mE_{\mv_3}=1,\qquad 
\#\mE^{\backslash\!\backslash}_{\mv_1\mv_2}=\#\mE^{\backslash\!\backslash}_{\mv_1\mv_3}=0,\ \#\mE^{\backslash\!\backslash}_{\mv_2\mv_3}=1,\qquad 
\#\mE^{\not\backslash\!\backslash}_{\mv_1\mv_2}=\#\mE^{\not\backslash\!\backslash}_{\mv_1\mv_3}=1,\ \#\mE^{\not\backslash\!\backslash}_{\mv_2\mv_3}=0
\]
and for $\mH_2$:
\[
\#\mE_{\mv_1}=\#\mE_{\mv_3}=\#\mE_{\mv_3}=1,\qquad 
\#\mE^{\backslash\!\backslash}_{\mv_1\mv_2}=\#\mE^{\backslash\!\backslash}_{\mv_1\mv_3}=\#\mE^{\backslash\!\backslash}_{\mv_2\mv_3}=1,\qquad 
\#\mE^{\not\backslash\!\backslash}_{\mv_1\mv_2}=\#\mE^{\not\backslash\!\backslash}_{\mv_1\mv_3}= \#\mE^{\not\backslash\!\backslash}_{\mv_2\mv_3}=0.
\]
\end{exa}

Given one or more directed hypergraphs, a few standard constructions can be performed.

\begin{defi}\label{def:union}
Given two directed hypergraphs $\mH_1=(\mV,\mE_1)$, $\mH_2=(\mV,\mE_2)$ with same vertex set, their \emph{union} (resp., \emph{intersection}) is the directed hypergraph with vertex set $\mV$ and hyperedge set $\mE_1\cup \mE_2$ (resp., $\mE_1\cap \mE_2$).
\end{defi}

\begin{defi} Given a directed hypergraph   $\mH$ with incidence matrix $\mathcal I$, the \emph{dual directed hypergraph to $\mH$} is the directed hypergraph whose incidence matrix is $\mathcal I^\top$. We denote it by $\mH^*$, and call the associated Laplacian, i.e., $\mathcal L_{\mH^*}:=\mathcal I^\top \mathcal I$, the \emph{dual Laplacian} to $\mathcal L$. 
\end{defi}

 Clearly, also this dual Laplacian is symmetric and positive semidefinite. It follows by elementary results in linear algebra that
\begin{equation}\label{eq:ranknul}
\dim \ker(\mathcal L_\mH)-\dim \ker(\mathcal L_{\mH^*})=\#\mV-\#\mE,
\end{equation}
see \cite[Corollary~14]{JosMul19}.
In the context of chemical physics, the dual Laplacian is associated with the so-called dual master equation~\cite{SriPolEsp23}.

\begin{rem}\label{rem:notation-after-josmul}
\begin{enumerate}
\item\label{remitem:josmul-1} Whenever a Laplacian on a directed graph is considered, its entries do not change if any edge is re-directed: therefore, each of the $2^{\#\mE}$ orientations of an undirected graph share the same Laplacian. If, however, one insists in considering graphs whose edges are not directed, one ends up with edges with two entrance endpoints and no exit endpoint (or vice versa, equivalently); this is the basic intuition behind the introduction of signless Laplacians \cite{Cve05}.

Things are similar in the case of hypergraph: As already observed in~\cite{JosMul19}, a direct computation yields for the entries of $\mathcal L$ the formula
\begin{equation}\label{eq:josmul-basic}
\begin{split}
\mathcal L_{\mv\mv}&=
\#\mE_\mv=\deg(\mv)\\
\mathcal L_{\mv\mw}&=\#\mE^{\backslash\!\backslash}_{\mv\mw}-\#\mE^{\not\backslash\!\backslash}_{\mv\mw}
\end{split}
\quad \hbox{for all }
\mv,\mw\in \mV;
\end{equation}
this shows that the Laplacian's off-diagonal entries (but not its diagonal entries!) may change whenever a single vertex is moved from $\metarg$ to $\mesour$, or vice versa. However, the Laplacian is invariant under swapping of hyperedge orientation, i.e., whenever  the roles of $\metarg$ and $\mesour$ are interchanged in a hyperedge $\me$.
\item\label{remitem:josmulrol}  Unlike $\mathcal L$, its dual Laplacian may well depend on the edge orientation of $\mH$; but, reasoning as in \autoref{remitem:josmul-1},  $\mathcal L_{\mH^*}$ is invariant under swapping of vertex role, i.e., interchanging the role of $\mesour$ and $\metarg$ for any hyperedge $\me$ does not affect $\mathcal L$.
\end{enumerate}
\end{rem}

Denote by
\[
\deg^{\mathrm{in}}(\mv)\quad\hbox{and}\quad \deg^{\mathrm{out}}(\mv)
\]
the number of hyperedges $\me$ for which $\mv\in \me_{\rm targ}$ and $\mv\in \me_{\rm sour}$, respectively.
In the special case of directed graphs, a classical double-counting argument yields the handshaking formula
\[
\sum_{\mv \in \mV}\deg^{\mathrm{in}}(\mv)=\sum_{\mv \in \mV}\deg^{\mathrm{out}}(\mv)=\#\mE.
\]
This can be immediately extended to directed hypergraphs as follows where, with a slight abuse of notation, we denote by $\deg(\me)$ the number of vertices that belong to $\me$, i.e., 
\[
\deg(\me):=\#\metarg+\#\mesour
\qquad\hbox{and accordingly}\qquad
\deg_{\max}(\mE):=\max_{\me\in\mE}\deg(\me).
\]

\begin{lemma}\label{lem:handsh}
Let $\mH=(\mV,\mE)$ be a directed hypergraph. Then
\[
\sum_{\mv \in \mV}\deg^{\mathrm{in}}(\mv)=\
\sum_{\me\in \mE}\#\metarg\quad\hbox{and}\quad
\sum_{\mv \in \mV}\deg^{\mathrm{out}}(\mv)=\
\sum_{\me\in \mE}\#\mesour.\]

In particular,
\[
\sum_{\mv\in\mV}\deg(\mv)
=\sum_{\me\in\mE}\deg(\me).
\]
\end{lemma}

In analogy with the case of graphs, we refer to the evolution equation
\[
\frac{\dd u}{\dd t}(t,\mv)=-\mathcal L u(t,\mv),\qquad t\ge 0,\ \mv \in \mV,
\]
as the \emph{heat flow} driven by $\mathcal L$; and to the semigroup $(\e^{-t\mathcal L})_{t\ge 0}$ that governs it as the \emph{heat semigroup} associated with $\mH$.

\begin{lemma}
Let $\mH$ be a directed hypergraph, and let $\mathcal L$ be the associated Laplacian. Then $(\e^{-t\mathcal L})_{t\ge 0}$ is a contractive, self-adjoint semigroup on $(\K^\mV,\|\cdot\|_2)$.
 \end{lemma}
 
 \begin{proof}
By construction, $\mathcal L$ is symmetric and positive semi-definite.
\end{proof}

\begin{exa}\label{exa:graphyper-compare}
In literature, an undirected hypergraph is often replaced by its \textit{2-section}, i.e., by a graph where each hyperedge $\me$ is replaced by a collection of all edges between any vertex in $\me$. It would be natural to adapt this notion to the setting of \textit{directed} hypergraphs replacing any hyperedge $\me$ by a collection of all edges between any vertex in $\mesour$ and any vertex in $\metarg$. Let us illustrate with an example that the heat flow on these two structures present fundamental differences.

Consider the directed hypergraph $\mH$ on three vertices, $\mV=\{\mv_1,\mv_2,\mv_3\}$ that is given by the signed incidence matrix

\begin{minipage}{3cm}
\begin{tikzpicture}[scale=.6, thick]

 \coordinate (v1) at (0,-.5);
 \coordinate (split) at (0,0.7);
 \coordinate (v2) at (-0.8,1.5);
 \coordinate (v3) at (0.8,1.5);

 \draw[thick] 
 (v1) .. controls +(0,0.4) and +(0,-0.4) .. (split);
 \draw[thick] 
 (split) .. controls +(-0.3,0.4) and +(0,-0.4) .. (v2);
 \draw[thick] 
 (split) .. controls +(0.3,0.4) and +(0,-0.4) .. (v3);


  \fill (v1) circle (2.5pt) ;
 \fill (v2) circle (2.5pt) ;
 \fill (v3) circle (2.5pt) ;

\node at (v1) [anchor=east] {$\mv_1$};
\node at (v2) [anchor=east] {$\mv_2$};
\node at (v3) [anchor=west] {$\mv_3$};

\end{tikzpicture}
\end{minipage}
\begin{minipage}{12.4cm}
\begin{equation}\label{eq:hyper3-a}
\mathcal I_\mH:=
\begin{pmatrix}
-1 \\ 1 \\ 1
\end{pmatrix},
\qquad \hbox{whence}\qquad
\mathcal L_\mH=
\begin{pmatrix}
1 & -1 & -1\\
-1 & 1 & 1\\
-1 & 1 & 1
\end{pmatrix}
\end{equation}
\end{minipage}

along with its 2-section $\mG$, which is given by the incidence matrix

\begin{minipage}{3cm}
\begin{tikzpicture}[scale=.6, thick]

 \coordinate (v1) at (0,-.5);
 \coordinate (split) at (0,0.7);
 \coordinate (v2) at (-0.8,1.5);
 \coordinate (v3) at (0.8,1.5);

 \draw[thick]  (v1) -- (v2);
 \draw[thick]  (v1) -- (v3);


 \fill (v1) circle (2.5pt)  ;
 \fill (v2) circle (2.5pt) ;
 \fill (v3) circle (2.5pt) ;

\node at (v1) [anchor=east] {$\mv_1$};
\node at (v2) [anchor=east] {$\mv_2$};
\node at (v3) [anchor=west] {$\mv_3$};

\end{tikzpicture}
\end{minipage}
\begin{minipage}{12.4cm}
\begin{equation}\label{eq:hyper3-b}
\mathcal I_\mG:=
\begin{pmatrix}
-1 & -1 \\ 1 & 0 \\ 0 & 1
\end{pmatrix}
\qquad \hbox{whence}\qquad
\mathcal L_\mG=
\begin{pmatrix}
2 & -1 & -1\\
-1 & 1 & 0\\
-1 & 0 & 1
\end{pmatrix}.
\end{equation}
\end{minipage}

One finds
\[
\e^{-t\mathcal L_\mH}=
\frac13
\begin{pmatrix}
2 + \e^{-3t} & 1-\e^{-3t}  & 1-\e^{-3t} \\
1 -\e^{-3t}  &  2+\e^{-3t}& \e^{-3t} - 1\\
1 -\e^{-3t} & \e^{-3t} - 1 & 2+\e^{-3t} 
 \end{pmatrix},\qquad t\ge 0,
\]
and
\[
\e^{-t\mathcal L_\mG}=
\frac{1}{6}
\begin{pmatrix}
4\e^{-3t} + 2 & 2-2\e^{-3t} & 2-2\e^{-3t} \\
2-2\e^{-3t}   & 2+3\e^{-t} + \e^{-3t}  & 2-3\e^{-t} + \e^{-3t}\\
2-2\e^{-3t}  & 2-3\e^{-t} + \e^{-3t} & 2+3\e^{-t} + \e^{-3t}
\end{pmatrix},\qquad t\ge 0.
\]

The behaviour of these semigroups is fundamentally different. For instance, imposing the initial condition $u(0,\cdot)=\begin{pmatrix}
0 & 1 & 0
\end{pmatrix}^\top$ we find the time evolution depicted in \autoref{fig:difference}: this shows that, unlike $(\e^{-t\mathcal L_\mG})$, the semigroup $(\e^{-t\mathcal L_\mH})$ is not positive, not $\infty$-contractive, not stochastic, not converging towards a consensus equilibrium -- indeed, converging towards a  \textit{dissensus} equilibrium.

\begin{figure}[ht]
\includegraphics[scale=.45]{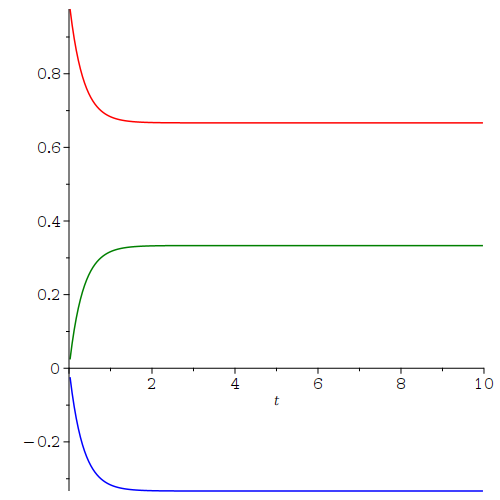} 
\includegraphics[scale=.45]{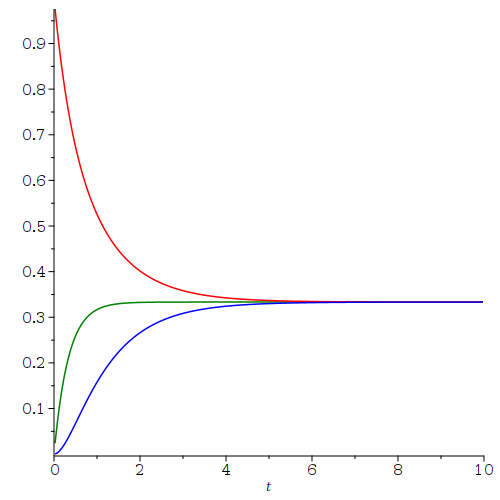} 
\caption{The time evolution driven by the semigroups $(\e^{-t\mathcal L_\mH})_{t\ge 0}$ (left) and $(\e^{-t\mathcal L_\mG})_{t\ge 0}$ (right) in \autoref{exa:graphyper-compare} for the initial condition 
$u_0=\begin{pmatrix}
0 & 1 & 0 
\end{pmatrix}^\top$. The green (resp., red, blue) curves depict the time evolution of $u(t,\mv)$ for $\mv=\mv_1$ (resp., $\mv=\mv_2$, $\mv=\mv_3$) and $t\ge 0$.
}
\label{fig:difference}.
\end{figure}
\end{exa}

Our aim in this article is to study the properties of $(\e^{-t\mathcal L})_{t\ge 0}$ in dependence of the structural properties of $\mH$. Let us mention the following, as an example.

\begin{prop}\label{prop:expstable-ex1}
Let $\mH'=(\mV,\mE_{\mH'})$ be the union of  a connected directed graph   $\mG=(\mV,\mE_{\mG})$ and a directed hypergraph  $\mH=(\mV,\mE_{\mH})$. Then the following assertions are equivalent:
\begin{enumerate}[(i)]
\item $\mathbf 1_{\#\mV}\in \ker \mathcal I^\top_{\mH}$;
\item $\ker \mathcal  L_{\mH'}$ is spanned by $\mathbf{1}_{\#\mV}$;
\item $\lim\limits_{t\to \infty}\e^{-t\mathcal L_{\mH'}}=\frac{1}{\#\mV}J_{\#\mV,\#\mV}$ in operator norm;
\item  $(\e^{-t\mathcal L_{\mH'}})_{t\ge 0}$ is not exponentially stable.
\end{enumerate}
\end{prop}

(Here and in the following
we denote by $J_{h,h}$ is the all-1 square matrix of size $h$.)

\begin{proof}
``(i)$\Rightarrow$ (ii)'' Clearly, $\ker \mathcal L_\mG =\ker\mathcal I^\top_\mG$ and $\ker \mathcal L_\mH =\ker\mathcal I^\top_\mH$.
 Since $\mG$ is connected, $\ker \mathcal L_\mG$ is spanned by the constant function $\mathbf{1}_{\#\mV}$.
Because the incidence matrix of $\mH'$ is 
\[
\mathcal I_{\mH'}=
\begin{pmatrix}
\begin{tabular}{c|c}
$\mathcal I_\mH$ & $\mathcal I_\mG$ 
\end{tabular} 
\end{pmatrix},
\]
$\mathbf 1\in \ker \mathcal I^\top_{\mH}$ implies that, in particular, $\mathbf 1\in \ker \mathcal I^\top_{\mH'}$, too. Furthermore, if $f\in \K^\mV$ is a function that lies in $\ker \mathcal L_{\mH'}=\ker \mathcal I^\top_{\mH'}$, a system
\begin{equation}\label{eq:ker-split}
\begin{cases}
\mathcal I_\mG^\top f=0,\\[4pt]
\mathcal I_\mH^\top f=0,
\end{cases}
\end{equation}
 of $\#\mE_{\mG}+\#\mE_{\mH}$ linear equations
must be satisfied: the first $\#\mE_{\mG}$ already force $f$ to be a constant function. This shows the claim.

``(ii)$\Rightarrow$(iii)'' This is an immediate consequence of the Spectral Theorem.

``(iii)$\Rightarrow$(iv)'' Obvious.

``(iv)$\Rightarrow$(i)'' Since the semigroup is not exponentially stable, by the Spectral Theorem $\lambda_1(\mathcal L_{\mH'})$ must vanish, i.e., $\ker \mathcal L_{\mH'}=\ker \mathcal I^\top_{\mH'}$ must be non-trivial. In other words, there is $f\in \K^\mV$ such that \eqref{eq:ker-split} is satisfied, i.e., such that both $\mathcal I^\top_\mG f=0$ (and hence $f$ is constant, since $\mG$ is connected) and $\mathcal I^\top_\mH f=0$.
\end{proof}


\begin{exa}\label{exa:prototyp}
Let us present three classes of directed hypergraphs that do \emph{not} satisfy the condition (i) for $\mH$ in \autoref{prop:expstable-ex1} and, hence, such that the semigroup generated by $-\mathcal L_{\mH'}$ \emph{is} exponentially stable. In view of the characteristic block matrix structure of the corresponding Laplacians, they will play a role in later Sections, too.

(Here and in the following for $h,k\in \mathbb N_0$ we denote by
\[
\mathbf{1}_{h}\qquad\hbox{and}\qquad J_{h, k}
\]
the all-1 vector of length $h$ and 
the all-1  $h\times k$ matrix, respectively.)
\begin{enumerate}[(1)]
\item\label{item:onehyperedge} Directed hypergraph $\mH$ that consists of an individual hyperedge (i.e., $\mE_{\mH}=\{\me\}$):  then, up to renumbering of the vertices, 
\[
\mathcal I_\mH =\begin{pmatrix}
\begin{tabular}{c}
$-\mathbf{1}_{d_-}$ \\
\hline 
 $\mathbf{1}_{d_+}$ 
\end{tabular} 
\end{pmatrix}
\qquad\hbox{whence}\qquad
\mathcal L_\mH=
\begin{pmatrix}
\begin{tabular}{c|c}
$J_{d_-, d_-}$ & $- J_{d_-, d_+}$ \\ 
\hline 
$-J_{d_+, d_-}$ & $ J_{d_+, d_+}$ \\ 
\end{tabular} 
\end{pmatrix}
\]
where we write $d_-:=\#\mesour$ source endpoints and $d_+:=\#\metarg$ target endpoints.
Accordingly, $\mathcal I_{\mH}^\top\mathbf{1}= \#\metarg-\#\mesour$, so the requirement in \autoref{prop:expstable-ex1} is satisfied if and only if $\#\mesour= \#\metarg$. Also, observe that the eigenvalues of $\mathcal L_\mH$ are $0^{\{\#\mV-1\}},\#\mV$ (\cite[Lemma~15]{JosMul19}) and it can be proved by induction
that its null space is spanned by the vectors 
\begin{equation}\label{eq:basisv-271-1}
\frac{e_i-e_{i+1}}{\sqrt{2}},\ i\in \{1,\ldots,d_- -1\},\qquad\hbox{whenever }d_-\ge 2,
\end{equation}
and
\begin{equation}\label{eq:basisv-271-2}
\frac{e_{d_- +j}-e_{d_- +j+1}}{\sqrt{2}},\ i\in \{1,\ldots,d_+ -1\},\qquad\hbox{whenever }d_+\ge 2,
\end{equation}
where $e_h$ denotes the $h$-th vector of the canonical basis of $\K^{\#\mV}$, along with
\begin{equation}\label{eq:basisv-271-3}
\begin{pmatrix}
\begin{tabular}{c}
$\sqrt{\frac{d_+}{\#\mV\cdot d_-}}\mathbf{1}_{d_-}$ \\[5pt]
\hline \\[-9pt]
 $\sqrt{\frac{d_-}{\#\mV \cdot d_+}}\mathbf{1}_{d_+}$ 
\end{tabular} 
\end{pmatrix}\qquad\hbox{whenever }d_-\ge 1,\ d_+\ge 1.
\end{equation}
The corresponding orthogonal projector is
\begin{equation}\label{eq:1vector-proj}
P_{\ker \mathcal L}=
\frac{1}{\#\mV}\begin{pmatrix}
\begin{tabular}{c|c}
$\#\mV\cdot\Id_{d_-}-J_{d_-, d_-}$ & $ J_{d_-, d_+}$ \\ 
\hline 
$J_{d_+, d_-}$ & $\#\mV\cdot \Id_{d_+}- J_{d_+, d_+}$ \\ 
\end{tabular} 
\end{pmatrix}.
\end{equation}

\item\label{item:hypersignless} Directed hypergraphs with same number $\#\mV$ of vertices and hyperedges, and such that $\metarg=\emptyset$ or $\mesour=\emptyset$  in each $\me\in \mE_{\mH}$: then, 
\[
\mathcal I_\mH =
\mathbf{1}_{\#\mV}
\;\;\hbox{or}\;\;
=-\mathbf{1}_{\#\mV}
\qquad\hbox{whence}\qquad
\mathcal L_\mH=
\#\mV\cdot J_{\#\mV, \#\mV}
\]
and, accordingly, $\mathcal I_{\mH}^\top\mathbf{1}= \#\mV\cdot \mathbf{1}$ or $=-\#\mV\cdot \mathbf{1}$, so the requirement in \autoref{prop:expstable-ex1} is always satisfied. Also, the eigenvalues of $\mathcal L_\mH$ are $0^{\{\#\mV-1\}},\#\mV^2$ and its null space is spanned by the vectors
\[
\varphi_i = \frac{1}{\sqrt{i(i+1)}} 
\begin{pmatrix}
\underbrace{1,1,\dots,1}_{i \text{ entries}}, -i, 0, \dots, 0
\end{pmatrix}^\top,
\qquad i = 1, 2, \dots, \#\mV-1.
\]
Accordingly, the corresponding orthogonal projector is $\Id_{\#\mV}-\frac{1}{\#\mV}J_{\#\mV, \#\mV}$.

\item\label{item:hyperrotat} Directed hypergraphs with same number $\#\mV$ of vertices and hyperedges, and such that $\mv_i$ is source in $\me_i$, $i=1,\ldots,\#\mV$, and target in all other hyperedges: then 
\[
\mathcal I_\mH=
\begin{pmatrix}
J_{\#\mV, \#\mV}-2\Id_{\#\mV}
\end{pmatrix}
\qquad\hbox{whence}\qquad
\mathcal L_\mH=4\Id_{\#\mV}+(\#\mV-4)J_{\#\mV, \#\mV}
\]
and, accordingly, $\mathcal I_{\mH}^\top \mathbf{1}= (\#\mV-2)\mathbf{1}$, so the requirement in \autoref{prop:expstable-ex1} is satisfied if and only if $\#\mV\ge 3$.

Also, the eigenvalues of $\mathcal L_\mH$ are $4^{\{\#\mV-1\}},(\#\mV-2)^2$  (\cite[Lemma~24]{JosMul19}). This exemplifies a major difference between graphs and directed hypergraphs; namely, in the latter case a strictly positive ground state energy may be achieved even without a potential or Dirichlet boundary conditions.

 If $\#\mV=4$, the multiplicity of the eigenvalue 4 corresponds to the dimension of the whole space, hence the orthogonal projector onto the corresponding eigenspace is the identity matrix $\Id_4$. For $\#\mV\ge 5$, one sees that the orthogonal projector onto the eigenspace corresponding to the eigenvalue 4 is $\Id_{\#\mV}-\frac{1}{\#\mV}J_{\#\mV,\#\mV}$.
\end{enumerate}

These directed hypergraphs also show that $\sum_{\mw\ne \mv}(\#\mE^{\backslash\!\backslash}_{\mv\mw}+\#\mE^{\not\backslash\!\backslash}_{\mv\mw})$ does not necessarily agree with $\deg(\mv)$. We also stress that $\mathcal L_{\mH}$ is diagonal (hence in particular not irreducible) if $\#\mV=4$.
\end{exa}

Let us also observe that
\[
|\mathcal L_{\mv\mw}|= |\#\mE^{\backslash\!\backslash}_{\mv\mw}-\#\mE^{\not\backslash\!\backslash}_{\mv\mw}|\le \#\mE^{\backslash\!\backslash}_{\mv\mw}+\#\mE^{\not\backslash\!\backslash}_{\mv\mw}=\deg(\mv)\qquad\hbox{for all }\mv\in\mV\hbox{ and all }\mw\ne \mv;
\]
summing over all $\mw\ne \mv$ we obtain the rough estimate
\begin{equation*}\label{eq:gerschrigho}
\sum_{\mw\ne\mv}|\mathcal L_{\mv\mw}|\le \left(\#\mV-1\right)\deg(\mv)\qquad \hbox{for all }\mv\in\mV,
\end{equation*}
which can be refined as follows.
\begin{lemma}\label{lem:diag-semidom}
Let $\mH$ be a directed hypergraph, and $\mathcal L$ be the associated Laplacian. Then
\begin{equation}\label{eq:gerschrigho-2}
\sum_{\substack{\mw\in\mV \\ \mw\neq \mv}}|\mathcal L_{\mv\mw}|\le \sum_{\me\in \mE_\mv} (\deg(\me) - 1)\qquad \hbox{for all }\mv\in\mV.
\end{equation}
\end{lemma}

\begin{proof}
By the triangle inequality we find for all $\mv\in\mV$
\[
\begin{split}
\sum_{\substack{\mw\in\mV \\ \mw\neq \mv}}|{\mathcal L}_{\mv\mw}|
= \sum_{\substack{\mw\in\mV \\ \mw\neq \mv}}\left|\sum_{\me\in \mE}\iota_{\mv\me}\iota_{\mw\me}\right|
&= \sum_{\substack{\mw\in\mV \\ \mw\neq \mv}}\left|\sum_{\me\in \mE_\mv}\iota_{\mv\me}\iota_{\mw\me}\right|\\
&\le \sum_{\substack{\mw\in\mV \\ \mw\neq \mv}}\sum_{\me\in \mE_\mv}|\iota_{\mv\me}\iota_{\mw\me}|
= \sum_{\me\in \mE_\mv} \sum_{\substack{\mw\in \mV\\ \mw\ne \mv}}|\iota_{\mw\me}|=\sum_{\me\in \mE_\mv}(\deg(\me) - 1).\end{split}
\]
This completes the proof.
\end{proof}

The proof shows that equality holds in \eqref{eq:gerschrigho-2} if and only if 
for each $\mv,\mw\in\mV$
\begin{itemize}
\item  either $\iota_{\mv\me}\iota_{\mw\me}\equiv -1$
for each $\me\in \mE_\mv\cap \mE_\mw$
\item or $\iota_{\mv\me}\iota_{\mw\me}\equiv 1$ for each $\me\in \mE_\mv\cap \mE_\mw$.
\end{itemize}
This is the case, in particular, if $\mH$ is a (directed multi)graph (i.e., both $\mesour,\metarg$ are singletons for each hyperedge $\me$), or if any two vertices are co-oriented in each hyperedge, but not only, as the example of the incidence matrix 
\[
\begin{pmatrix}
-1 & -1\\
1 & 1\\
0 & 1
\end{pmatrix}
\]
shows.

\begin{rem}\label{rem:misc}
\begin{enumerate}[(1)]
\item In \cite{JosMul19}, the condition $\me_{\rm sour}\cap \me_{\rm targ}=\emptyset$ is sometimes relaxed: the case $\me_{\rm sour}\cap \me_{\rm targ}\not=\emptyset$ is then incorporated in the general theory by letting $\iota_{\mv\me}:=0$ for each such $\mv\in \me_{\rm sour}\cap \me_{\rm targ}$, and $\mH$ is then called a \textit{chemical hypergraph}. In this context, the elements of $\me_{\rm sour}\cap \me_{\rm targ}$ are interpreted as catalyst of a chemical reaction. However interesting for modelling purposes, allowing for such catalyst vertices does hardly change the mathematical theory: for the sake of simplicity, catalysts are not considered in this article.

\item Let us call a directed hypergraph $\mH$ \emph{proper} if both $h\ne 0$ and $k\ne 0$, i.e., if for all $\me\in \mE$ $\me_{\rm sour}\ne\emptyset\ne \me_{\rm targ}$.
For modelling purposes in chemical reactions, the pre-condition that the hypergraph is proper is frequently imposed in the literature since \cite{JosMul19}, but for our purposes it seems to be an unnecessary restriction.

Indeed, if each hyperedge $\me$ has the property that either $\mesour=\emptyset$ or $\metarg=\emptyset$ for all $\me\in\mE$ (as e.g.\ in \autoref{exa:prototyp}.(\ref{item:hypersignless})), then 
$\mathcal I$ is the hypergraph counterpart on the signless incidence matrix of a graph and it is natural to regard $\mathcal L$ as the hypergraph generalization of a graph's signless Laplacian. This is all the more natural since each \textit{undirected} hypergraph can be regarded as an incidence structure, and vice versa. In other words, considering a directed hypergraph $\mH$ with $ \metarg=\emptyset$ or  $ \mesour=\emptyset$  for all $\me\in\mE$ -- and hence imposing that any two vertices are co-oriented in all hyperedges -- merely amounts to considering the case where $\mH$ is an incidence structure. 
\item A few crucial differences between Laplacians on (directed) graphs and directed hypergraphs can be illustrated drawing on the 
\autoref{exa:prototyp}:
\begin{itemize}
\item Possibly, $\mathcal L_{\mv\mw}\ge 0$ even if the vertices $\mv,\mw$ belong to a common hyperedge (\autoref{exa:prototyp}.(\ref{item:hyperrotat}), $\#\mV=4$).

\item Possibly, $\sum\limits_{\mw\ne \mv}|\mathcal L_{\mv\mw}|\not\le \deg(\mv)$,
i.e., $\mathcal L$ is generally not a diagonally dominant matrix (\autoref{exa:prototyp}.(\ref{item:onehyperedge})).
\item Possibly, $\mathcal L\mathbf{1}\ne 0$ even when $0$ is an eigenvalue (\autoref{exa:prototyp}.(\ref{item:onehyperedge}), $\#\metarg\ne\#\mesour$).
\end{itemize}
Indeed, appropriate modification of the Laplacian that allow for weightings of the vertices in such a way that constant functions are always in the null space of the Laplacian, and hence  stationary solutions for the heat flow, have been devised in~\cite{FazTenLuk24,ShaTiaZha25}.
\end{enumerate}
\end{rem}



\section{Spectral Theory}\label{sec:spectheor}

Given any directed hypergraph $\mH=(\mV,\mE)$, the spectrum of its Laplacian consists of $\#\mV$ real eigenvalues, which we denote by
\[
\lambda_1(\mH),\ldots,\lambda_{\#\mV}(\mH),
\]
in non-decreasing order.

\begin{lemma}\label{lem:gersh0}
Let $\mH$ be a directed hypergraph. Then,  for both the Laplacian $\mathcal L_{\mH}$ and the dual Laplacian $\mathcal L_{\mH^*}$, all eigenvalues  are nonnegative and the sum of all eigenvalues agrees with $\sum\limits_{\me\in\mE}\deg(\me)=\sum\limits_{\mv\in\mV}\deg(\mv)$.

If $\#\mV\ge \#\mE$, then the dimension of the null space of the Laplacian $\mathcal L_{\mH}$ is at least $\#\mV-\#\mE$. If $\#\mE\ge \#\mV$, then the dimension of the null space of the dual Laplacian $\mathcal L_{\mH^*}$ is at least $\#\mE-\#\mV$.
\end{lemma}

We stress that the first two assertions do not depend on the sets $\mesour,\metarg$, but only on $\me$: i.e., the sign of all eigenvalues and the value of their sum are two properties that are shared by all directed versions of the same undirected hypergraph.

\begin{proof}
 All eigenvalues are nonnegative, since $\mathcal L$ is positive semidefinite. The second assertion follows computing the trace of $\mathcal L$, using \eqref{eq:josmul-basic} and \autoref{lem:handsh}. Because the non-zero eigenvalues of the Laplacian $\mathcal L_{\mH}=\mathcal I\mathcal I^\top$ and of the dual Laplacian $\mathcal L_{\mH^*}=\mathcal I^\top \mathcal I$ agree, these assertions hold for the eigenvalues of the dual Laplacian, too.
 
 Finally, the last assertion follows from the Rank--Nullity theorem: denoting by $r$ the rank of $\mathcal I$, we find that $\dim\ker(\mathcal L)=\#\mV-r\ge \#\mV-\#\mE$, and likewise for the dual Laplacian.
\end{proof}

It is possible to obtain rough bounds on the eigenvalues of Laplacians on directed hypergraphs.

\begin{prop}\label{lem:gersh1}
Let $\mH=(\mV,\mE)$ be a directed hypergraph. Then the following assertions hold.

\begin{enumerate}[(1)]
\item 
Each eigenvalue $\lambda$ of $\mathcal L$ satisfies
\[
\deg_{\min}(\mV)-\max_{\mv\in\mV} \sum_{\substack{\mw\in\mV\\ \mw\ne \mv}}|\#\mE^{\backslash\!\backslash}_{\mv\mw}-\#\mE^{\not\backslash\!\backslash}_{\mv\mw}|\le\lambda \le \deg_{\max}(\mV)+\max_{\mv\in\mV}\sum_{\substack{\mw\in\mV\\ \mw\ne \mv}}|\#\mE^{\backslash\!\backslash}_{\mv\mw}-\#\mE^{\not\backslash\!\backslash}_{\mv\mw}|
\]
as well as
\[
\begin{split}
\lambda \le\deg_{\max}(\mE)+ \max_{\me\in\mE} \sum_{\substack{\mf\in\mE\\ \mf\ne \me}}\left|\#(\metarg\cap \mftarg)+\#(\mesour\cap \mfsour)-\#(\metarg\cap \mfsour)-\#(\mesour\cap \metarg)\right|.
\end{split}
\]
In particular, if for each $\mv\in \mV$
\[
\deg(\mv)> \sum_{\substack{\mw\in\mV\\ \mw\ne \mv}}|\#\mE^{\backslash\!\backslash}_{\mv\mw}-\#\mE^{\not\backslash\!\backslash}_{\mv\mw}|
\]
then 0 is not an eigenvalue and $(\e^{-t\mathcal L})_{t\ge 0}$ is exponentially stable.

\item Each eigenvalue $\lambda$ of $\mathcal L$ satisfies
\begin{equation}\label{eq:firstgershc}
2\deg_{\min}(\mV)-\max_{\mv\in \mV} \sum_{\me\in \mE_\mv}\deg(\me)\le \lambda \le \max_{\mv\in \mV} \sum_{\me\in \mE_\mv}\deg(\me)
\end{equation}
as well as
\begin{equation}\label{eq:secondgershc}
\lambda \le
\max_{\me\in \mE} \sum_{\mv\in \me} \deg(\mv).
\end{equation}
In particular, if for each $\mv\in \mV$
\[
2\deg(\mv)> \sum_{\me\in \mE_\mv} \deg(\me)
\]
then 0 is not an eigenvalue and $(\e^{-t\mathcal L})_{t\ge 0}$ is exponentially stable.
\end{enumerate}
\end{prop}

Observe that \eqref{eq:firstgershc} returns the classical estimate
\begin{equation}\label{eq:firstgershc-graph}
\lambda\le 2\deg_{\max}(\mV)
\end{equation}
whenever $\mathcal L$ is the Laplacian or the signless Laplacian on a graph $\mG$.

\begin{proof}

(1) By Ger\v{s}gorin’s theorem, every eigenvalue $\lambda$ of $\mathcal L=\mathcal I\mathcal I^\top$ lies in an interval of centre $\deg(\mv)$ and radius $\sum_{\substack{\mw\in\mV\\ \mw\ne\mv}}|\mathcal L_{\mv\mw}|$ for some $\mv\in \mV$. Now, the absolute value of the off-diagonal $\mv-\mw$ entries of $\mathcal L$ are
\[
\left|\sum_{\me\in \mE}\iota_{\mv\me}\iota_{\mw\me}\right|=|\#\mE^{\backslash\!\backslash}_{\mv\mw}-\#\mE^{\not\backslash\!\backslash}_{\mv\mw}|
\]
and the first claim follows.

Let us now denote by $\mu_{\me\mf}$ the $\me-\mf$ entry of the dual Laplacian $\mathcal I^\top\mathcal I$.
Each column $\me$ of $\mathcal I$ has $\deg(\me)$ non-zero entries $\pm 1$, so
\(
\mu_{\me\me} = \deg(\me).
\)
Moreover, for $\mf\neq \me$,
\(
\mu_{\me\mf} = \sum\limits_{\mv\in \mV} \iota_{\mv \me}\iota_{\mv \mf}
\)
is non-zero only when $\me$ and $\mf$ share vertices.
Again by Ger\v{s}gorin's Theorem, each eigenvalue of $\mathcal I^\top\mathcal I$ lies in an interval 
\[
\hbox{centred at }\deg(\me)\quad \hbox{ with radius }\sum_{\substack{\mf\in \mE\\ \mf\ne \me}}|\mu_{\me\mf}|,\quad \hbox{for some }\me\in \mE.
\]
Also, the absolute value of the off-diagonal $\me-\mf$ entries of the dual Laplacian $\mathcal I^\top\mathcal I$ are
\[
|\mu_{\me\mf}|=\left| \sum_{\mf\in\mV}\iota_{\mv\me}\iota_{\mv\mf}\right|=\left|\#(\metarg\cap \mftarg)+\#(\mesour\cap \mfsour)-\#(\metarg\cap \mfsour)-\#(\mesour\cap \metarg)\right|.
\]
and the second bound follows, too, since $\mathcal I\mathcal I^\top$ and $\mathcal I^\top\mathcal I$ have the same non-zero eigenvalues.

(2) It immediately follows combining Ger\v{s}gorin's Theorem and \autoref{lem:diag-semidom} that each eigenvalue satisfies
\begin{equation*}
\begin{split}
2\deg(\mv) - \sum_{\me\in \mE_\mv}\deg(\me)
&=
\deg(\mv) - \sum_{\me\in \mE_\mv}(\deg(\me) - 1) \\
& \le \lambda \le \deg(\mv) + \sum_{\me\in \mE_\mv}(\deg(\me) - 1)
= \sum_{\me\in \mE_\mv} \deg(\me)\qquad\hbox{for some }\mv\in\mV
\end{split}
\end{equation*}
Taking the maximum over all vertices yields the upper bound in \eqref{eq:firstgershc}.

To prove the second bound, observe that for each $\me\in\mE$,
\begin{align*}
\sum_{\mf\neq \me} |\mu_{\me\mf}|
&\le \sum_{\mv\in \me} (\deg(\mv) - 1)
\end{align*}
therefore applying Ger\v{s}gorin's Theorem we find that each eigenvalue $\lambda$ of the dual Laplacian 
(and hence each non-zero eigenvalue of the Laplacian) satisfies
\begin{equation*}
 \lambda \le \deg(\me) + \sum_{\mv\in \me} (\deg(\mv) - 1)
= \sum_{\mv\in \me} \deg(\mv).
\end{equation*}
Taking the maximum over all edges yields \eqref{eq:secondgershc}.
\end{proof}

As an immediate consequence of \autoref{rem:notation-after-josmul} and \autoref{rem:misc}.(\ref{remitem:josmulrol}) and the co-spectrality of $\mathcal L_{\mH}$ and $\mathcal L_{\mH^*}$ away from 0 we obtain the following version of \cite[Lemma~49]{JosMul19}.

\begin{lemma}
Let $\mH$ be a directed hypergraph.
The non-zero eigenvalues of neither the Laplacian $\mathcal L_{\mH}$ nor the dual Laplacian $\mathcal L_{\mH^*}$ change under swapping of hyperedge orientation or swapping of vertex role.
\end{lemma}

%
%


By the Spectral Theorem, if the eigenvalues of $\mathcal L$ satisfy $\lambda_2>\lambda_1=0$, then $\lambda_2$ gives the rate of convergence to equilibrium of $(\e^{-t\mathcal L})_{t\ge 0}$.
Let us formulate an estimate for $\lambda_2$, where we adopt the notation introduced in \autoref{rem:notation-after-josmul}. In the following we call a directed hypergraph \textit{equipotent} if
\[
\#\metarg =\#\mesour\qquad \hbox{for all }\me\in\mE.
\]

\begin{lemma}\label{lem:a10charac}
Let $\mH$ be a directed hypergraph. Consider the following conditions:
\begin{enumerate}[(a)]
\item\label{item:equipot-equipot} The directed hypergraph $\mH$ is equipotent.
\item\label{item:nullpace-equipot} $\mathcal L\mathbf{1}=0$.
\item\label{item:combinatorial-equipot} $\deg(\mv)+\sum_{\mw\ne \mv}\#\mE^{\backslash\!\backslash}_{\mv\mw}=\sum_{\mw\ne \mv}\#\mE^{\not\backslash\!\backslash}_{\mv\mw}$ for all $\mv\in \mV$.
\item\label{item:fiedler-equipot} $\lambda_2 \le \frac{\#\mV}{\#\mV-1}\deg_{\min}(\mV)$.
\end{enumerate}
Then (\ref{item:equipot-equipot})$\Leftrightarrow$(\ref{item:nullpace-equipot})$\Leftrightarrow$(\ref{item:combinatorial-equipot})$\Rightarrow$(\ref{item:fiedler-equipot}).
\end{lemma} 
 
Clearly, oriented graphs are equipotent, as $\#\metarg =\#\mesour\equiv 1$, and indeed (\ref{item:combinatorial-equipot}) is satisfied with $\mE^{\backslash\!\backslash}_{\mv\mw}\equiv \emptyset$, $\deg(\mv)\equiv \sum_{\mw\ne \mv}\#\mE^{\not\backslash\!\backslash}_{\mv\mw}$.
 Also, the directed hypergraphs in \autoref{exa:prototyp}.(\ref{item:hyperrotat}) show that the implication (\ref{item:fiedler-equipot})$\Rightarrow$(\ref{item:nullpace-equipot}) 
 does \emph{not} hold.

\begin{proof}
It was already observed in \cite[Lemma~20]{JosMul19} that (\ref{item:equipot-equipot})$\Rightarrow$(\ref{item:nullpace-equipot}), whereas the converse implication holds because $\Ker \mathcal L=\Ker \mathcal I^\top$ and $\mathcal I^\top\mathbf{1}=\#\metarg -\#\mesour$.
The equivalence (\ref{item:nullpace-equipot})$\Leftrightarrow$(\ref{item:combinatorial-equipot}) follows from \eqref{eq:josmul-basic}.
The implication (\ref{item:nullpace-equipot})$\Rightarrow$(\ref{item:fiedler-equipot}) can be proved exactly like \cite[3.5]{Fie73}.
\end{proof}

\begin{exa} 
Finer estimates can be achieved by applying the theory developed in~\cite{DalMugSch15}. 
Consider two directed hypergraphs from the class discussed in~\autoref{exa:prototyp}.(\ref{item:onehyperedge}), for simplicity in the case $\#\mV=3$: we let
\[
\mathcal I_{1}:=\begin{pmatrix}
-1 \\ 1 \\ 1
\end{pmatrix}\qquad\hbox{and}\qquad 
\mathcal I_{2}:=\begin{pmatrix}
1 \\ 1 \\ 1
\end{pmatrix}:
\]
we already know from \autoref{exa:prototyp}.(\ref{item:onehyperedge}) that the spectrum of both $\mathcal L_1,\mathcal{L}_2$ is $0^{\{2\}},3$. 

We want to compute 
\[
R_{31}(\lambda)=\frac{a_{13}^2 + |a_{13}a_{23}+a_{12}(\lambda-a_{33})|}{(\lambda-a_{11})(\lambda-a_{33})}\qquad\hbox{and}\qquad
R_{32}(\lambda):=\frac{a_{23}^2 + |a_{13}a_{23}+a_{12}(\lambda-a_{33})|}{(\lambda-a_{22})(\lambda-a_{33})}
\]
since, by \cite[Theorem 4.1]{DalMugSch15}, and using the notation therein,
\[
\sigma(A)\subset S_3(A):=\{a_{11},a_{22},a_{33}\}\cup \{\lambda\ne \{a_{11},a_{33}\}:R_{31}(\lambda)\ge 1\}\cup 
\{\lambda\ne \{a_{22},a_{33}\}:R_{32}(\lambda)\ge 1\}.
\]
Now, for $\mathcal L_2$ we find
\[
R_{31}(\lambda)=R_{32}(\lambda)=\frac{1+|\lambda|}{(1-\lambda)^2}
\]
whence
\[
\sigma(\mathcal L_2)\subset S_3(\mathcal L_2)=[0,3].
\]
Concerning $\mathcal L_1$, 
\[
R_{21}(\lambda)=
R_{31}(\lambda)=\frac{1+|\lambda|}{(1-\lambda)^2},\qquad 
R_{32}(\lambda)=\frac{1+|\lambda-2|}{(1-\lambda)^2}
\]
so by \cite[Corollary~4.5]{DalMugSch15},
\[
\sigma(\mathcal L_1)\subset S_1(\mathcal L_1)\subset [0,3].
\]
This methods refines the bounds based on Ger\v{s}gorin's Theorem, which could only deliver the inclusions
$\sigma(\mathcal L_1),\sigma(\mathcal L_2)\subset[-1,3]$.
\end{exa}

\subsection{Hypergraph perturbation of graphs}

Because the eigenvalues are given by the Rayleigh quotient 
\[
f\mapsto \frac{\|\mathcal I^\top f\|_{\ell^2(\mE)}^2}{\|f\|_{\ell^2(\mV)}^2},\qquad f:\mV\to \R,
\]
we immediately obtain the following, as edge deletion leads to a smaller numerator in the numerator of the Rayleigh quotient.

\begin{lemma}\label{lem:hyperfiedler}
Let $\mH$ be a directed hypergraph. Then the lowest eigenvalue $\lambda_1$ of $\mathcal L$ does not increase upon deleting hyperedges of $\mH$.
\end{lemma}

In particular, given two hypergraphs $\mH_1,\mH_2$, the $k$-th eigenvalue of the Laplacian of their union $\mH_1\cup \mH_2$ is not smaller than the $k$-the eigenvalue of the either $\mathcal L_{\mH_1}$ or $\mathcal L_{\mH_2}$.

\begin{lemma}\label{lem:hyperfiedler-hoeld}
Let $\mH=(\mV,\mE_\mH)$ be an equipotent directed hypergraph such that $\deg(\me)=2m$ for all $\me\in\mE$ and some $m\in\N$. Then
\[
\lambda_1(\mH)\le m\lambda_1(\mG),
\]
where $\mG=(\mV,\mE_\mG)$ is any graph obtained replacing each hyperedge $\me\in \mE_\mH$ by $m$ edges in such a way that each element of $\mesour$ is connected with precisely one element of $\metarg$.
\end{lemma}

\begin{proof}
Take $\me\in \mE_\mH$ and consider an arbitrary numeration $\metarg=\{\mv_1,\ldots,\mv_m\}$ and $\mesour=\{\mw_1,\ldots,\mw_m\}$. Then
\[
|(\mathcal I^\top_\mH f)(\me)|^2=\left| \sum_{\mv\in\metarg}f(\mv)-
 \sum_{\mw\in\mesour}f(\mw)\right|^2=\left|\sum_{i=1}^m(f(\mv_i)-f(\mw_i)) \right|^2\le m\sum_{i=1}^m|f(\mv_i)-f(\mw_i)|^2,
\]
whence
\[
\|\mathcal I^\top_\mH f\|_{\ell^2(\mE_\mH)}^2\le m\|\mathcal I^\top_\mG f\|_{\ell^2(\mE_\mG)}^2:
\]
this yields the claim.
\end{proof}

Let us now focus on a special class of surgical results that consider a hypergraph Laplacian that arises as perturbation of the Laplacian or signless Laplacian of a graph.

\begin{lemma}\label{lem:spectralgeomhyper}
Let the directed hypergraphs $\mH$ be the union of a graph $\mG$ and a hypergraph consisting of a unique hyperedge $\me$ that contains all the vertices of $\mG$ and such that any two vertices are co-oriented in $\me$. 

Then (counting multiplicity) the spectrum of $\mathcal L_\mH$ consists of 0 (with multiplicity $c-1$), of $\#\mV$, and of all non-zero eigenvalues of $\mathcal L_\mG$; here $c$ denotes the number of connected components of $\mG$. The lowest eigenvalue $\lambda_{\min}$ is $\#\mV$ if and only if $\mG$ is the complete graph, and in this case the orthogonal projector onto the corresponding eigenspace is the identity. Otherwise, $\lambda_{\min}<\#\mV$ and the orthogonal projector is not a positive matrix.
\end{lemma}

The setting of \autoref{lem:spectralgeomhyper-signless}: the fact that $0$ is not an eigenvalue of $\mathcal L_\mH$ whenever $\mG$ is connected was shown to hold under more general assumptions on the hyperedge $\me$ in \cite[Lemma~23]{JosMul19}.

\begin{proof}
The incidence matrix of $\mH$ is 
\[
\mathcal I_{\mH}=
\begin{pmatrix}
\begin{tabular}{c|c}
$\mathbf{1}_{\#\mV}$ & $\mathcal I_\mG$ 
\end{tabular} 
\end{pmatrix}
\qquad\hbox{whence}\qquad
\mathcal I^\top_\mH =\begin{pmatrix}
\begin{tabular}{c}
$\mathbf{1}^\top_{\#\mV}$ \\[3pt]
\hline  \\[-10pt]
 $\mathcal{I}^\top_{\mG}$ 
\end{tabular} 
\end{pmatrix}
\]
and therefore the $\#\mV\times \#\mV$ matrix $\mathcal L_\mH$ is given by
\begin{equation}\label{eq:thisisprec}
\mathcal L_{\mH}=\mathcal I_\mH \mathcal I_\mH^\top
=\mathbf{1}_{\#\mV}\mathbf{1}_{\#\mV}^\top +
\mathcal I_\mG \mathcal I_\mG^\top
=
J_{\#\mV,\#\mV}+\mathcal L_\mG.
\end{equation}

Now, observe that $J_{\#\mV,\#\mV},\mathcal L_\mG$ share the eigenvector $\mathbf{1}_{\#\mV}$, with $J_{\#\mV,\#\mV}\mathbf{1}_{\#\mV}=\#\mV\cdot\mathbf{1}_{\#\mV}$ and $\mathcal L_\mG\mathbf{1}_{\#\mV}=0$: so $\#\mV $ is an eigenvalue of $\mathcal L_\mH$. If, furthermore, $\mV_1,\ldots,\mV_c$ denote the vertex sets of the $c$ connected components of $\mG$, then whenever $c>1$
\[
\#\mV_2 \mathbb{1}_{\mV_1}-\#\mV_1 \mathbb{1}_{\mV_2},\ldots,\#\mV_c \mathbb{1}_{\mV_1}-\#\mV_1 \mathbb{1}_{\mV_c}
\]
yield $c-1$ further eigenvectors of $\mathcal L_\mH$, since all of them lie in the null space of both $\mathcal L_\mG$ (because the characteristic functions $\mathbb{1}_{\mV_i}$ of the of the connected components span its null space) and $J_{\#\mV,\#\mV}$, and hence of $\mathcal L_\mH$.

Finally,  $\mathcal L_\mG$ has eigenvectors that are not connected component-wise constant: each of them is orthogonal to $\mathbf{1}_{\#\mV}$ and, hence, lies in the null space of $J_{\#\mV,\#\mV}$, leading to $\#\mV-c$ further eigenvalues of $\mathcal L_\mH$, all of them non-zero. Summing up, the spectrum of $\mathcal L_\mH$ contains 
\begin{itemize}
\item $\#\mV$, 
\item 0 with multiplicity $c-1$ (if $c>1$), and 
\item all $\#\mV-c$ nonzero eigenvalues of $\mathcal L_\mG$.
\end{itemize}
Counting multiplicity, these are $\#\mV$ eigenvalues and, hence, exhaust the spectrum of $\mathcal L_\mH$.

Now, it is known (\cite[3.6 and 3.10]{Fie73}) that the lowest nonzero eigenvalue of the connected graph $\mG$ is $\#\mV$ if and only if $\mG$ is complete, and $\le \#\mV-2$ otherwise.

 In the former case ($\mG$ complete), the spectrum of $\mathcal L_\mG$ is $0,\#\mV^{\{\#\mV-1\}}$ and we conclude that the spectrum of $\mathcal L_\mH$ is $\#\mV^{\{\#\mV\}}$, i.e., the corresponding eigenspace is $\K^\mV$ and accordingly the orthogonal projector onto it is $\Id_{\#\mV}$.
 
In the latter case ($\mG$ not complete) 
%
the orthogonal projector $P$ onto the eigenspace corresponding to the lowest eigenvalue $\lambda_{\min}$ is not a positive matrix: if it were, by Perron--Frobenius theory \cite[Theorem~8.3.1]{HorJoh91} $P$ would have a positive eigenvector $\xi$ for the dominant eigenvalue (which is necessarily 1, since $P$ is a non-trivial projector), i.e., at least one eigenvector $\xi$ of $\mathcal L_\mH$ corresponding to $\lambda^*$ could be chosen positive, contradicting its orthogonality to the eigenvector $\mathbf{1}_{\#\mV}$. 
%
\end{proof}

Let us now turn to perturbations of signless Laplacians. As already mentioned, we regard a signless Laplacian as the Laplacian of a graph with no anti-oriented vertices: using its known spectral properties \cite[Section~7.8]{CveRowSim10} we deduce the following.

\begin{lemma}\label{lem:spectralgeomhyper-signless}
Let the directed hypergraphs $\mH$ be the union of a graph $\mG$ that is bipartite with respect to $\widetilde{\mV}_-\cup \widetilde{\mV}_+$ and such that all of its vertices are co-oriented in all of its edges;
and a hypergraph consisting of a unique hyperedge $\me$ that $\mesour=\widetilde{\mV}_-$ and $\metarg=\widetilde{\mV}_+$. 

Then (counting multiplicity) the spectrum of $\mathcal L_\mH$ consists of 0 (with multiplicity $c-1$), of $\#\mV$, and of all non-zero eigenvalues of $\mathcal L_\mG$; here $c$ denotes the number of connected components of $\mG$. 
\end{lemma}
\begin{proof}
Let us denote by $\mathbb{1}_{\widetilde{\mV}_\pm}$ the characteristic function of the vertex set $\widetilde{\mV}_\pm$, and let $\widetilde{\mathbf{1}}_{\#\mV}:=\mathbb{1}_{\widetilde{\mV}_+}-\mathbb{1}_{\widetilde{\mV}_-}$. Then, the incidence matrix of $\mH$ is 
\[
\mathcal I_{\mH}=
\begin{pmatrix}
\begin{tabular}{c|c}
$\widetilde{\mathbf{1}}_{\#\mV}$ & $\mathcal I_\mG$ 
\end{tabular} 
\end{pmatrix}
\qquad\hbox{whence}\qquad
\mathcal I^\top_\mH =\begin{pmatrix}
\begin{tabular}{c}
$\widetilde{\mathbf{1}}_{\#\mV}^\top$ \\[3pt]
\hline  \\[-10pt]
 $\mathcal{I}^\top_{\mG}$ 
\end{tabular} 
\end{pmatrix}
\]
 and therefore the $\#\mV\times \#\mV$ matrix $\mathcal L_\mH$ is given by
\begin{equation}\label{eq:thisisprec-2}
\mathcal L_{\mH}=\mathcal I_\mH \mathcal I_\mH^\top
=(\mathbb{1}_{\widetilde{\mV}_+}-\mathbb{1}_{\widetilde{\mV}_-})\cdot(\mathbb{1}_{\widetilde{\mV}_+}-\mathbb{1}_{\widetilde{\mV}_-})^\top+
\mathcal I_\mG \mathcal I_\mG^\top
=
\widetilde{J}_{+,-}+\mathcal L_\mG,
\end{equation}
where
\[
\widetilde{J}_{+,-}:=\begin{pmatrix}
\begin{tabular}{c|c}
$J_{\#\widetilde{\mV}_+,\#\widetilde{\mV}_+}$ & $-J_{\#\widetilde{\mV}_+,\#\widetilde{\mV}_-}$ \\ 
\hline 
$-J_{\#\widetilde{\mV}_-,\#\widetilde{\mV}_+}$ & $J_{\#\widetilde{\mV}_-,\#\widetilde{\mV}_-}$ \\ 
\end{tabular} 
\end{pmatrix}:
\]
observe that whenever $\#\widetilde{\mV}_+\ge 1$ and $\#\widetilde{\mV}_-\ge 1$, the eigenvalues of $\widetilde{J}_{+,-}$ are $0^{\{\#\mV-1\}},\#\mV$.
Also, we stress that the assumption that all vertices are co-oriented in each edge of $\mG$ amounts to saying that $\mathcal L_\mG$ is the signless Laplacian of $\mG$.

Now, observe that $\widetilde{J}_{+,-},\mathcal L_\mG$ share the eigenvector $\widetilde{\mathbf{1}}_{\#\mV}$, with $\widetilde{J}_{+,-}\widetilde{\mathbf{1}}_{\#\mV}=\#\mV\cdot\widetilde{\mathbf{1}}_{\#\mV}$ and $\mathcal L_\mG\widetilde{\mathbf{1}}_{\#\mV}=0$: so $\#\mV $ is an eigenvalue of $\mathcal L_\mH$. 

If, furthermore, $c>1$ and $\mV_1,\ldots,\mV_c$ denote the vertex sets of the $c$ connected components of $\mG$, then 
\[
\widetilde{\mathbf{1}}_{\#\mV}\left(\#\mV_2 \mathbb{1}_{\mV_1}-\#\mV_1 \mathbb{1}_{\mV_2}\right),\ldots,
\widetilde{\mathbf{1}}_{\#\mV}\left(\#\mV_c \mathbb{1}_{\mV_1}-\#\mV_1 \mathbb{1}_{\mV_c}\right)
\]
(that is, appropriate linear combinations of the restrictions to the connected components of the characteristic functions of the sets $\widetilde{\mV}_\pm$)
yield $c-1$ further eigenvectors of $\mathcal L_\mH$, since all of them lie in the null space of both $\mathcal L_\mG$ (because the restrictions of the alternating function $\widetilde{\mathbf{1}}_{\#\mV}$ to each connected component of $\mG$ span the null space of the signless Laplacian $\mathcal L_\mG$) and $\widetilde{J}_{+,-}$, and hence of $\mathcal L_\mH$.

Finally,  $\mathcal L_\mG$ has eigenvectors that are not connected component-wise constant: each of them is orthogonal to $\widetilde{\mathbf{1}}_{\#\mV}$ and, hence, lies in the null space of $\widetilde{J}_{+,-}$, leading to $\#\mV-c$ further non-zero eigenvalues of $\mathcal L_\mH$. Summing up, the spectrum of $\mathcal L_\mH$ contains 
\begin{itemize}
\item $\#\mV$, 
\item 0 with multiplicity $c-1$ (if $c>1$), and 
\item all $\#\mV-c$ nonzero eigenvalues of $\mathcal L_\mG$.
\end{itemize}
Counting multiplicity, these are $\#\mV$ eigenvalues and, hence, exhaust the spectrum of $\mathcal L_\mH$.
\end{proof}

\begin{lemma}\label{lem:spectralgeomhyper-2}
Let the directed hypergraphs $\mH$ be the union of a connected graph $\mG$ on an even number of vertices, $\#\mV=2h$,  and a hypergraph consisting of a unique hyperedge $\me$ that contains all the vertices of $\mG$ and such that  $\#\mesour=\#\metarg=h$.

Then the lowest eigenvalue of $\mathcal L_\mH$ is 0, it is simple, and the corresponding eigenspace is spanned by a strictly positive vector.
\end{lemma}

\begin{proof}
This is an immediate consequence of \autoref{prop:expstable-ex1}, since the hyperedge is equipotent and, hence,  $\mathbf{1}_{\#\mV}$ lies in the null space of transpose of its incidence matrix; accordingly the null space of $\mathcal L_{\mH}$ is spanned by $\frac{1}{\sqrt{\#\mV}}\mathbf{1}_{\#\mV}$ and, therefore, the orthogonal projector onto null space is $\frac{1}{\#\mV}J_{\#\mV,\#\mV}$.
\end{proof}

In view of the results on the eigenvalue 0, we immediately deduce the following from the Spectral Theorem.

\begin{cor}\label{cor:frogexponential}
Under the assumptions of either \autoref{lem:spectralgeomhyper} or \autoref{lem:spectralgeomhyper-signless}, $(\e^{-t\mathcal L_\mH})_{t\ge 0}$ is exponentially stable if and only if $\mG$ is connected.

Under the assumption of \autoref{lem:spectralgeomhyper-2}, $(\e^{-t\mathcal L_\mH})_{t\ge 0}$ is not exponentially stable.
\end{cor}

It seems to be not easy to apply further surgical principles to derive eigenvalue estimates or comparison principles and, therefore, deduce information about the large-time behaviour of the heat flow. Indeed, the spectral theory of a directed hypergraph is much subtler than that of a graph: this can be seen already with a toy model.

\begin{exa}
Let us relax the notion of incidence matrix of a directed hypergraph and, inspired by the notion of stoichiometric matrix in chemical physics \cite{PolEsp14}, introduce the generalised incidence matrix
\[
\mathcal I_\varepsilon=
\begin{pmatrix}
-1\\
1\\
\varepsilon
\end{pmatrix},\qquad \varepsilon\ge 0.
\]
If $\varepsilon=0$, we are considering a disconnected graph consisting of three vertices $\mv_1,\mv_2,\mv_3$ and an edge connecting $\mv_1,\mv_2$: its eigenvalues are $0,0,2$.

If, however, $\varepsilon>0$, we are ``switching on'' a hypergraph-type behaviour by adding an interaction between $\mv_3$ and the existing edge.
Then one sees that the spectrum becomes $0,0,2+\varepsilon^2$. However, two of the three corresponding eigenvectors do not converge as $\varepsilon\to 0$, i.e., the eigenvectors for $\varepsilon=0$ cannot be simply transplanted into those for $\varepsilon>0$; indeed, 
\[
\hbox{it is \underline{not} true that}\quad \|\mathcal I^\top_\varepsilon f\|_{\ell^2(\mE)}^2\ge \|\mathcal I^\top_0 f\|_{\ell^2(\mE)}^2\quad\hbox{for all }f=(f_1,f_2,f_3)^\top\in \K^\mV.
\]
Therefore, the monotonicity of the spectrum with respect to $\varepsilon$ cannot be seen at the level of the Rayleigh quotient alone.
\end{exa}

%
%

%

\section{Order theoretical aspects}\label{sec:ordertheor}

\subsection{Positivity and eventual positivity}

A positive semigroup $(T(t))_{t\ge 0}$ that is $\infty$-contractive is said to be \emph{sub-Markovian}: this is equivalent to the invariance of the order interval $[-\mathbf{1},\mathbf{1}]$ under the heat flow, i.e., to $T(t)\mathbf{1}\le \mathbf{1}$ for all $t\ge 0$. If, additionally, $T(t)\mathbf{1}=\mathbf{1}$ for all $t\ge 0$, then the semigroup is said to be \textit{Markovian}, and its adjoint is said to be \textit{stochastic}.

The Markovian property of the heat flow on graphs has been regarded as one of its most remarkable properties, ever since~\cite{BeuDen59}.
On hypergraphs, it is convenient to relax the notion of positivity of the heat flow. Let us recall the following.

\begin{defi}
A semigroup $(T(t))_{t\ge 0}$ on a finite dimensional Banach lattice $E$ with positive cone $E_+$ is said to be 
\begin{itemize}
\item \emph{asymptotically positive} if $\mathrm{dist } (\e^{t(A-s(A))}f, E_+)\to 0$ as $t\to\infty$ for all $f\ge 0$;
\item \emph{positive} if $T(t)f\ge 0$ for all $t\ge 0$ and all $f\ge 0$; 
\item \emph{irreducible} if is positive and $T(t)f\ge c\mathbb 1$ for some $c>0$, all $t\ge 0$ and all $f\ge 0$; 
\item \emph{eventually positive} if there is $t_0>0$ such that $T(t)f\ge 0$ for some $c>0$, all $t\ge t_0$ and all $f\ge 0$; 
\item \emph{eventually irreducible} if it is eventually positive and there is $t_0>0$ such that $T(t)f\ge c\mathbb 1$ for some $c>0$, all $t\ge 0$ and all $f\ge 0$.
\end{itemize}
\end{defi}

In finite dimensional Hilbert lattices, the first Beurling--Deny condition that characterises a semigroup's positivity can be simplified: it is well-known (see, e.g., \cite[Theorem~7.1]{BatKraRha17}) that the semigroup $(\e^{-tA})_{t\ge 0}$ generated by a real matrix $-A$ is positive if and only if $A$ is a Z-matrix, i.e., it has real entries and its off-diagonal entries are non-negative. Therefore, in view of \eqref{eq:josmul-basic} we immediately obtain the following.

\begin{prop}\label{prop:posit}
Let $\mH$ be a directed hypergraph. Then
$(\e^{-t\mathcal L})_{t\ge 0}$ is positive if and only if
\begin{equation}\label{eq:M-matr-hypercond}
\#\mE^{\backslash\!\backslash}_{\mv\mw}\le \#\mE^{\not\backslash\!\backslash}_{\mv\mw}\qquad\hbox{for all }\mv,\mw\in \mV.
\end{equation}
\end{prop}

\begin{cor}
Given a directed hypergraph $\mH=(\mV,\mE_\mH)$, there exists a multigraph $\mG=(\mV,\mE_\mG)$ such that $(\e^{-t\mathcal L_{\mH'}})_{t\ge 0}$ is positive, where $\mH'$ is the union of $\mH$ and $\mG$.
\end{cor}
\begin{proof}
If $\mathcal L_{\mv\mw}>0$,
adding an edge between $\mv,\mw$ raises $\deg(\mv),\deg(\mw)$ by 1 and lowers $\mathcal L_{\mv\mw},\mathcal L_{\mw\mv}$ by one.
\end{proof}
\autoref{exa:graphyper-compare} shows that the perturbed semigroup $(\e^{-t\mathcal L_{\mH'}})_{t\ge 0}$ need not be sub-Markovian, though.

\begin{rem}
\begin{enumerate}[(1)]
\item Because by definition $\#\mE^{\backslash\!\backslash}_{\mv\mw}$ (resp., $\#\mE^{\not\backslash\!\backslash}_{\mv\mw}$) is the cardinality of the set
\[
\{\me\in \mE: \iota_{\mv\me}\iota_{\mw}=+1 \}
\qquad\hbox{(resp., of }\{\me\in \mE: \iota_{\mv\me}\iota_{\mw}=-1 \}\ ),
\]
the condition \eqref{eq:M-matr-hypercond} can be checked looking at the incidence matrix alone, without any need to explicitly compute the Laplace.

\item
In the particular case of a graph, \autoref{prop:posit} returns the known facts that minus the Laplacian of an oriented graph (the usual graph Laplacian) generates a positive semigroup; whereas minus the Laplacian of an undirected graph (the signless Laplacian) generates a non-positive semigroup.

\item
On a graph $\mG$, $\mathcal L$ and hence $(\e^{-t\mathcal L})_{t\ge 0}$ are irreducible if and only if $\mG$ is connected. In the case of directed hypergraphs, there seems to be no universally accepted notion of connectedness, least so one that is linked to the behaviour of the heat flow.
\end{enumerate}
\end{rem}

Several notions related to eventual positivity can be characterised as follows, by~\cite[Proposition~6.2.7, Theorem~7.3.3 and Theorem~10.2.1]{Glu16}.

\begin{lemma}\label{lem:glu}
Let  $A$ be a self-adjoint operator on a finite dimensional Hilbert lattice. Denote by $P$ the orthogonal projector onto the eigenspace $E_{s(A)}$ associated with the spectral bound $s(A)$.
 Then the following assertions hold.
\begin{enumerate}[(1)]
\item\label{item:isem25eventpos} If $(\e^{tA})_{t\ge 0}$ is eventually positive, then $E_{s(A)}$ contains a positive vector.
\item\label{item:gluasympos} $(\e^{tA})_{t\ge 0}$ is asymptotically positive if and only if $P$ is a positive matrix;
\item\label{item:glueventirr} $(\e^{tA})_{t\ge 0}$ is eventually irreducible if and only if $s(A)$ is a simple eigenvalue and $E_{s(A)}$ is spanned by a strictly positive vector.
\end{enumerate}
\end{lemma}

The heat flow on a directed hypergraph may or may not be eventually irreducible. Let us provide a necessary condition.

\begin{lemma}\label{lem:ranknul-noevenirred}
Let $\mH$ be a directed hypergraph. 
If $\#\mV\ge \#\mE+2$, then the semigroup $(\e^{-t\mathcal L})_{t\ge 0}$ is not eventually irreducible. If $\#\mE\ge \#\mV+2$, then the dual semigroup $(\e^{-t\mathcal L^*})_{t\ge 0}$ is not eventually irreducible.
\end{lemma}

\begin{proof}
By \autoref{lem:gersh0}, the dimension of the null space of the Laplacian $\mathcal L_{\mH}$ is at least $\#\mV-\#\mE$, hence the claim follows from \autoref{lem:glu} if $\#\mV-\#\mE\ge 2$. The second assertions can be proved likewise.
\end{proof}

It is known that given a finite set $\mV\ne \emptyset$, then $(\e^{-t\mathcal L_{\mG_1}})_{t\ge 0}$ does not dominate $(\e^{-t\mathcal L_{\mG_2}})_{t\ge 0}$ for any two given, different graphs $\mG_1,\mG_2$ with same vertex set $\mV$. 
Things are different when it comes to directed hypergraphs.

%
%

\begin{prop}\label{prop:arorgluunion}
Under the assumptions of \autoref{lem:spectralgeomhyper}, the semigroup $(\e^{-t\mathcal L_\mH})_{t\ge 0}$ is positive if and only if $\mG$ is the complete graph: in this case, $(\e^{-t\mathcal L_\mH})_{t\ge 0}$ is dominated by $(\e^{-t\mathcal L_\mG})_{t\ge 0}$.

If $\mG$ is not the complete graph, then $(\e^{-t\mathcal L_\mH})_{t\ge 0}$ is not even asymptotically positive.  However, if additionally $\mG$ is connected, then $(\e^{-t\mathcal L_\mH})_{t\ge 0}$ is eventually dominated by $(\e^{-t\mathcal L_\mG})_{t\ge 0}$.
\end{prop}

\begin{proof}
As seen in the proof of \autoref{lem:spectralgeomhyper},
\(
\mathcal L_{\mH}=\mathcal L_\mG+J_{\#\mV,\#\mV}
\).
This shows that $\mathcal L_\mH$ is a (real) Z-matrix if and only if all off-diagonal entries of $\mathcal L_\mG$ are not larger than $-1$: this is the case if and only if $\mG$ is the complete graph.
This proves the first claim.

As pointed out in \cite[Section~4.2]{GluMug21}, provided $\mathcal L_\mG,\mathcal L_\mH$ are both a Z-matrices, the claimed domination property is equivalent to $-\mathcal L_{\mG}$ dominating entrywise $-\mathcal L_{\mH}$, i.e., $\mathcal L_{\mH}-\mathcal L_{\mG}$ being a positive matrix: this is clearly true, since this difference is $J_{\#\mV,\#\mV}$.

If $\mG$ is the complete graph, then the lowest eigenvalue of $\lambda_\mH$ is $\#\mV$ with multiplicity $\#\mV$, so the orthogonal projector onto the corresponding eigenspace is the identity. If, however, $\mG$ is \textit{not} the complete graph, then by \autoref{lem:spectralgeomhyper} the lowest eigenvalue is associated with an orthogonal projector that is not a positive matrix, hence by~\autoref{lem:glu} $(\e^{-t\mathcal L_\mH})_{t\ge 0}$ is not asymptotically positive.

If $\mG$ is connected, then the last assertion follows from \cite[Theorem 3.1]{AroGlu23} and \autoref{cor:frogexponential}, since $(\e^{-t\mathcal L_\mG})_{t\ge 0}$ is positive and irreducible, while $(\e^{-t\mathcal L_\mH})_{t\ge 0}$ is exponentially stable.
\end{proof}

Let us now consider a similar but different case, that proposed in \autoref{lem:spectralgeomhyper-2}: in that case, $(\e^{-t\mathcal L_\mH})_{t\ge 0}$ may or may not be positive, as the cases when $\mG$ is a path graph or a star show, respectively. Is there any domination relation between the semigroups generated by $-\mathcal L_\mH,-\mathcal L_\mG$?
Let $\varphi$ be the vector representing the incidence matrix corresponding to the hyperedge $\me$. Then 
\[
\mathcal I_{\mH}=
\begin{pmatrix}
\begin{tabular}{c|c}
$\varphi$ & $\mathcal I_\mG$ 
\end{tabular} 
\end{pmatrix}
\qquad\hbox{whence}\qquad
\mathcal I_\mH\mathcal I_\mH^\top =\varphi \varphi^\top + \mathcal I_\mG\mathcal I_\mG^\top 
\]
Now, half of the entries of the matrix $\varphi \varphi^\top$ are $1$ and half of the entries are $-1$:  accordingly neither $\mathcal L_\mG$ dominates $\mathcal L_\mH$, nor vice versa.
Indeed, the following stronger assertion holds.

\begin{prop}
Under the assumptions of \autoref{lem:spectralgeomhyper-2},
 the semigroup $(\e^{-t\mathcal L_\mH})_{t\ge 0}$ is eventually irreducible. Neither does $(\e^{-t\mathcal L_\mG})_{t\ge 0}$ eventually dominate $(\e^{-t\mathcal L_\mH})_{t\ge 0}$, nor vice versa.
\end{prop}

\begin{proof}
The first assertion follows from \autoref{lem:spectralgeomhyper-2} and \autoref{lem:glu}, the second from \cite[Theorem~3.7]{GluMug21}.
\end{proof}

\subsection{$\infty$-contractivity and eventual $\infty$-contractivity}

Just like a semigroup's positivity is equivalent to the invariance of the order interval $[0,\infty)$ under the heat flow, its contractivity with respect to $\|\cdot\|_\infty$ amounts to the  invariance of the order interval $[-1,1]$. The following characterisation of the contractivity of $(\e^{-t\mathcal L})_{t\ge 0}$ with respect to $\|\cdot\|_\infty$ in terms of diagonal dominance is an immediate consequence of \cite[Lemma~6.1]{Mug07} and \eqref{eq:josmul-basic}.

\begin{lemma}\label{lem:mug07}
Let $\mH$ be a directed hypergraph. Then $(\e^{-t\mathcal L})_{t\ge 0}$ is $\infty$-contractive if and only if 
\begin{equation}\label{eq:diag-dom}
\sum\limits_{\mw\ne \mv}\left|\#\mE^{\backslash\!\backslash}_{\mv\mw}- \#\mE^{\not\backslash\!\backslash}_{\mv\mw}\right|\le \deg(\mv)\qquad \hbox{for all }\mv\in\mV.
\end{equation}
\end{lemma}

Because $\mathcal L$ is hermitian, by duality we obtain that \eqref{eq:diag-dom} holds if and only if $(\e^{-t\mathcal L})_{t\ge 0}$ is $1$-contractive, too. We  consider a stronger property, namely stochasticity, at the end of this section.

\begin{rem}
\begin{enumerate}
\item The condition \eqref{eq:diag-dom} can be written as
\begin{equation}\label{eq:diag-dom-alt}
\sum\limits_{\mw\ne \mv}\left|\sum_{\me\in\mE}\iota_{\mv\me}\iota_{\mw\me}\right|\le \deg(\mv)\qquad \hbox{for all }\mv\in\mV
\end{equation}
and can hence be checked looking at the incidence matrix alone, without any need to explicitly compute the Laplace.

\item 
Recall that a matrix $A$ is said to have \emph{factor width $k$} if $k$ is the smallest
integer such that $A$ can be factorised as $A = V V^\top$ for some real (rectangular) matrix $V$ each of whose columns contains at most $k$ non-zero entries: for example, Laplacians and signless Laplacians have factor width at most 2 (taking as $V$ the signed and signless incidence matrix, respectively) and so have \textit{adjacency matrices of generalised line graphs} \cite[\S~7.3]{Mug18} -- in fact exactly 2 because  matrices with factor width 1 are  precisely the diagonal matrices with nonnegative entries.
It is tempting to regard such real matrix $V$ with at most 2 non-zero entries on each column as a generalised incidence matrix of a  graph.
Clearly, both Laplacians and signless Laplacians of graphs satisfy \eqref{eq:diag-dom} and are, thus, diagonally dominant.

Conversely, it was proved in \cite[Theorems 8 and 9]{BomChePar05} that the matrices $A=(a_{ij})$ with factor width at most 2 are precisely those symmetric matrices with nonnegative diagonal that are diagonally dominant in a generalised sense, namely
\begin{equation}\label{eq:diag-dom-gener}
\sum\limits_{j\ne i}|a_{ij}|\xi_j \le |a_{ii}|\xi_i \qquad \hbox{for all }i
\end{equation}
holds for some $\xi\in \K^\mV$, $\xi>0$.

\end{enumerate}

\end{rem}

Therefore: $(\e^{-t\mathcal L})_{t\ge 0}$ is generally not $\infty$-contractive.
We therefore introduced the following, inspired by the relaxed notions of positivity discussed in the last 10 years (see~\cite{Glu16} and references therein).

\begin{defi}
A semigroup $(\e^{t(A-s(A))})_{t\ge 0}$ on a finite dimensional Banach lattice $E$ with unit order interval $[-\mathbf{1},\mathbf{1}]$ is said to be 
\begin{itemize}
\item \emph{asymptotically $\infty$-contractive} if $\mathrm{dist } (\e^{t(A-s(A))}f, [-\mathbf{1},\mathbf{1}])\to 0$ as $t\to\infty$ for all $|f|\le \mathbf{1}$;
\item \emph{eventually $\infty$-contractive} if there is $t_0>0$ such that $|\e^{tA} f|\le \mathbf{1}$ for all $t\ge t_0$ and all $|f|\le \mathbf{1}$.
\end{itemize}
\end{defi}

\begin{rem}\label{rem:inftyposobv}
\begin{enumerate}

\item\label{item:inftyposobv-nag}
The implication
\[
\e^{tA}\hbox{ is eventually positive and }A\mathbf{1}\le 0 \quad \Rightarrow \quad \e^{tA}\hbox{ is eventually $\infty$-contractive}
\]
follows immediately from \cite[Lemma B-III-2.1]{Nag86} (or rather, from the proof of the first implication therein).

\item\label{item:inftyposobv-equiv} 
Also, let $A$ be a finite Hermitian matrix. If the spectral bound $s(A)$ of $A$ is strictly negative, hence $\|\e^{tA}f\|_2\to 0$ as $t\to\infty$ for all $f$, then it follows from the equivalence of all norms in the finite dimensional setting that $\|\e^{tA}f\|_\infty\le 1$ for $t$ large enough and all $f$, i.e., $(\e^{tA})_{t\ge0}$ is eventually strongly $\infty$-contractive, too.

\item\label{item:inftycontr-neithernor} If $A$ is any complex number with strictly positive real part, then $(\e^{tA})_{t\ge 0}$ is asymptotically $\infty$-contractive but not eventually $\infty$-contractive.
Conversely, we will encounter at the end of Section~\ref{sec:fano} two examples $(\e^{-\mathcal L^\pm_2})_{t\ge 0}$ of semigroups that are  $\infty$-contractive, but not asymptotically $\infty$-contractive.
\end{enumerate}
\end{rem}
%
Let us  formulate the following.

\begin{prop}\label{prop:eventlinftycontr}
Let $(X,\mu)$ be a measure space and $A$ be a self-adjoint, negative semidefinite operator with compact resolvent on $L^2(X,\mu)$. Denote by $P$ the orthogonal projector onto the eigenspace $E_{s(A)}$ associated with the spectral radius $s(A)$.
 Then the following assertions are equivalent:

\begin{enumerate}[(i)]
\item $P$ is $\infty$-contractive;
\item $(\e^{tA})_{t\ge 0}$ is asymptotically $\infty$-contractive.
\end{enumerate}
\end{prop}
\begin{proof}
It follows from the Spectral Theorem that $\lim_{t\to \infty}\e^{t(A-s(A))}f=Pf$ for all $f\in L^2(X,\mu)$.

$(i)\Rightarrow(ii)$ Let $f\in L^2(X,\mu)$ with $-\mathbf{1}\le f\le \mathbf{1}$, so that by assumption $-\mathbf{1}\le Pf\le \mathbf{1}$, too. Then, for all $t\ge 0$
\[
\mathrm{dist } (\e^{t(A-s(A))}f, [-\mathbf{1},\mathbf{1}])\le \|\e^{t(A-s(A))}f- Pf\|+\mathrm{dist}(Pf,[-\mathbf{1},\mathbf{1}])=\|\e^{t(A-s(A))}f- Pf\|\stackrel{t\to \infty}{\to}0.
\]

$(ii)\Rightarrow(i)$ 
Let $f\in L^2(X,\mu)$ with $-\mathbf{1}\le f\le \mathbf{1}$. Then, for all $t\ge 0$
\[
\begin{split}
\mathrm{dist } (Pf, [-\mathbf{1},\mathbf{1}])
&\le \|\e^{t(A-s(A))}f- Pf\|+\mathrm{dist } (\e^{t(A-s(A))}f, [-\mathbf{1},\mathbf{1}])\stackrel{t\to \infty}{\to}0,
\end{split}
\]
as the second addend of the RHS converges to 0 by assumption.
\end{proof}

In the special case of finite dimensional Hilbert spaces, we obtain the following.

\begin{cor}\label{cor:eventlinftycontr-matr}
Under the assumptions of \autoref{prop:eventlinftycontr}, let $X$ be finite, i.e., $\dim L^2(X,\mu)=:N<\infty$.
Let $(\varphi_k)_{1\le k\le N}$ be an orthonormal basis of eigenvectors of $A$ with associated eigenvalues $(\lambda_k)_{1\le k\le N}$, and let $J_0:=\{k\in \{1,\ldots,N\}:\lambda_k=s(A)\}$.  
Then $(\e^{tA})_{t\ge 0}$ is asymptotically $\infty$-contractive if and only if
\begin{equation}\label{eq:infty-contr-nonsimple}
\max_{1 \le i \le N} \sum_{j=1}^N \left|\sum_{k\in J_0} (\varphi_k)_i (\varphi_k)_j\right| \le 1.
\end{equation}
\end{cor}
In particular, whenever the eigenvalue $s(A)$ is simple, say $J_0:=\{1\}$, then \eqref{eq:infty-contr-nonsimple} simplifies to
\begin{equation}\label{eq:ellinftysi}
\|\varphi_1\|_1\|\varphi_1\|_\infty\le 1.
\end{equation}
\begin{proof}
The orthogonal projector $P$ onto $E_{s(A)}$ is $\varphi_1\varphi_1^\top+\ldots+\varphi_k\varphi_k^\top$: in view of the well-known criterion for $\infty$-contractivity of a matrix, see \cite[\S~5.6.5]{HorJoh90},  the $\infty$-contractivity of $P$ onto $E_{s(A)}$ is therefore equivalent to \eqref{eq:infty-contr-nonsimple}.
\end{proof}

\begin{lemma}\label{lem:hypersignl}
Let $\mH$ be a directed hypergraph. Let all vertices of $\mH$ be co-oriented in each hyperedge. Then $(\e^{-t\mathcal L})_{t\ge 0}$ is $\infty$-contractive if and only if each hyperedge contains at most two vertices.
\end{lemma}

Observe that the signless Laplacian of a graph is a special case of Laplacian of a directed hypergraph all of whose vertices are co-oriented.

\begin{proof}
If all vertices are co-oriented, then the Laplacian can be represented as
\[
\mathcal D+\mathcal A
\]
where $\mathcal D$ is the degree matrix and $\mathcal A$ is, by \eqref{eq:josmul-basic}, a positive matrix: in particular, $\mathcal L$ is positive, too.

If each hyperedge contains at most two vertices, then each off-diagonal entry of $\mathcal A$ is at most 1 and $\sum\limits_{\mw\ne\mv}\mathcal A_{\mv\mw}\le \deg(\mv)$.
Then $\infty$-contractivity of $(\e^{-t\mathcal L})$ is an immediate consequence of \autoref{lem:mug07}.

Let, on the contrary, $\mH$ contain at least one hyperedge $\me$ with more than two vertices, say $\mv,\mw,\mz$. Then the diagonal entry $\mathcal L_{\mv\mv}=\deg(\mv)$ counts the hyperedge $\me$ once, but on the same line there are at least two entries $\mathcal L_{\mv\mw},\mathcal L_{\mv\mz}\ge 1$:
this shows that~\eqref{eq:diag-dom} is not satisfied.
\end{proof}

%

%

\begin{prop}\label{prop:greiner}
Let $\mH$ be an equipotent directed hypergraph.
Then for the semigroup $(\e^{-t\mathcal L})_{t\ge 0}$ positivity (resp., eventual positivity) is equivalent to $\infty$-contractivity (resp., eventual $\infty$-contractivity).
\end{prop}

\begin{proof}
Because $\mathcal L{\mathbf{1}}=0$, the assertion follows from~\cite[Lemma B-III-2.1]{Nag86}.
 \end{proof}

Recall that a positive semigroup $(T(t))_{t\ge 0}$ on $L^1(X,\mu)$ is said to be \textit{stochastic} if
$\|T(t)f\|_1=\|f\|_1$ for all $t\ge 0$ and all $0\le f\in L^1(X,\mu)$.
Combining \autoref{prop:posit}, \autoref{lem:a10charac} and \autoref{prop:greiner} we immediately obtain the following.

\begin{cor}\label{cor:stochequi}
Let $\mH$ be a directed hypergraph. Then $(\e^{-t\mathcal L})_{t\ge 0}$ is a stochastic semigroup if and only if $\mH$ is equipotent and additionally
\[
\#\mE^{\backslash\!\backslash}_{\mv\mw}\le \#\mE^{\not\backslash\!\backslash}_{\mv\mw}\qquad \hbox{for all }\mv,\mw\in \mV.
\]
\end{cor}

\subsection{Sub-hypergraphs and boundary conditions}

The theory of subgraphs with boundary conditions is classical in graph theory, at least since \cite{Fri93,Chu97}.
In this section we  elaborate on the natural extensions of these ideas to the case of hypergraphs. Comparable ideas leading to sub-hypergraphs have been discussed in~\cite[Section~4]{ShaTiaZha25} and~\cite[Section~4]{KurMulNab25}, and properties of diffusion equations on sub-hypergraphs have been discussed in~\cite{FukIkeUch24}.

\begin{defi}
Given a directed hypergraph $\mH=(\mV,\mE)$, let $\mV'\subset \mV$ and denote $\mV'_0:=\mV'\setminus \mV$: we refer to the operator $\mathcal L^{\mV',\mathrm{D}}_\mH:= P_{\mV'} \mathcal L_\mG P_{\mV'}$ as the \textit{Laplacian on $\mH$ with Dirichlet conditions on $\mV'_0$}. 
\end{defi}
Clearly, $\mathcal L^{\mV',\mathrm{D}}_\mH$ agrees with the part of $\mathcal L_\mH$ in the closed subspace $\{f\in \C^\mV:f(\mv)=0\hbox{ for all }\mv\in \mV'_0\}$.

\begin{exa}\label{exa:dirichlet-markovian?}
Imposing Dirichlet conditions on a subset of the vertex set may change the behaviour of the heat flow. 
\begin{enumerate}[(1)]
\item\label{item:dirichmark1} Consider, for example, the directed hypergraph $\mH$ with vertex set $\mV=\{\mv_1,\mv_2,\mv_3,\mv_4\}$ consisting of only one hyperedge, see Figure~\ref{fig:hyper-4-balanced}.

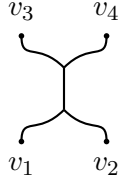
\begin{figure}[ht]
\begin{tikzpicture}[scale=.7, thick,
 arrowmid/.style={
 postaction={decorate},
 decoration={
 markings,
 }
 }
]

 \coordinate (v1) at (-0.8,0);
 \coordinate (v2) at (0.8,0);
 \coordinate (merge) at (0,0.6);
 \coordinate (split) at (0,1.4);
 \coordinate (v3) at (-0.8,2);
 \coordinate (v4) at (0.8,2);

 \draw[thick]
 (v1) .. controls +(0,0.5) and +(-0.5,-0.5) .. (merge);
 \draw[thick]
 (v2) .. controls +(0,0.5) and +(0.5,-0.5) .. (merge);

 \draw[thick, arrowmid]
 (merge) -- (split);

 \draw[thick]
 (split) .. controls +(-0.5,0.5) and +(0,-0.5) .. (v3);
 \draw[thick]
 (split) .. controls +(0.5,0.5) and +(0,-0.5) .. (v4);

 \fill (v1) circle (1.5pt) node[below=2pt] {$v_1$};
 \fill (v2) circle (1.5pt) node[below=2pt] {$v_2$};
 \fill (v3) circle (1.5pt) node[above=2pt] {$v_3$};
 \fill (v4) circle (1.5pt) node[above=2pt] {$v_4$};

\end{tikzpicture}\caption{The directed hypergraph consisting of one hyperedge and four vertices discussed in~\autoref{exa:dirichlet-markovian?}}
\label{fig:hyper-4-balanced}
\end{figure}

We  consider three cases:
\begin{itemize}
\item $\mV'=\{\mv_1,\mv_2,\mv_3,\mv_4\}$, whence $\mV_0=\emptyset$;
\item $\mV'=\{\mv_1,\mv_2,\mv_3\}$, whence $\mV_0=\{\mv_4\}$;
\item $\mV'=\{\mv_2,\mv_3\}$, whence $\mV_0=\{\mv_1,\mv_4\}$.
\end{itemize}
A direct computation shows that the corresponding Laplacians $\mathcal L^{\mV',\mathrm{D}}_\mG$ is, in these cases,
\[
\begin{pmatrix}
1 & 1 & -1 & -1\\
1 & 1 & -1 & -1\\
-1 & -1 & 1 & 1\\
-1 & -1 & 1 & 1\\
\end{pmatrix},\quad 
\begin{pmatrix}
1 & 1 & -1 & 0\\
1 & 1 & -1 & 0\\
-1 & -1 & 1 & 0\\
0 & 0 & 0 & 0\\
\end{pmatrix},\quad 
\begin{pmatrix}
0 & 0 & 0 & 0\\
0 & 1 & -1 & 0\\
0 & -1 & 1 & 0\\
0 & 0 & 0 & 0\\
\end{pmatrix}.
\]
This shows that imposing Dirichlet condition may ``restore'' the positivity properties that are typical of graph Laplacians: indeed, the semigroup $(-\e^{-t\mathcal L^{\mV',\mathrm{D}}}_\mH)_{t\ge 0}$ is Markovian and (whenever restricted to the relevant space of functions supported on $\{\mv_2,\mv_3\}$) irreducible in the third case, but not in the second case -- in fact, not even eventually irreducible, as 0 is a non-simple eigenvalue even after restricting $\mathcal L^{\mV',\mathrm{D}}$ to the relevant space of functions supported on $\{\mv_1,\mv_2,\mv_3\}$ (indeed, it has multiplicity 2).

\item If the directed hypergraph $\mH$ with vertex set $\mV=\{\mv_1,\mv_2,\mv_3,\mv_4,\mv_5\}$ in \autoref{fig:hyper-4-balanced-additional-2} is considered, instead, then $(\e^{-t\mathcal L^{\mV',\mathrm{D}}}_\mH)_{t\ge 0}$ is Markovian in both cases
\begin{itemize}
\item $\mV'=\{\mv_2,\mv_3,\mv_5\}$, whence $\mV_0=\{\mv_1,\mv_4\}$;
\item $\mV'=\{\mv_2,\mv_3\}$, whence $\mV_0=\{\mv_1,\mv_4,\mv_5\}$.
\end{itemize}
\begin{figure}[ht]
\begin{tikzpicture}[scale=.7, thick,
 arrowmid/.style={
 postaction={decorate},
 decoration={
 markings,
 }
 }
]

 \coordinate (v1) at (-0.8,0);
 \coordinate (v2) at (0.8,0);
 \coordinate (merge) at (0,0.6);
 \coordinate (split) at (0,1.4);
 \coordinate (v3) at (-0.8,2);
 \coordinate (v4) at (0.8,2);
 \coordinate (v5) at (-1.8,2);

 \draw[thick]
 (v1) .. controls +(0,0.5) and +(-0.5,-0.5) .. (merge);
 \draw[thick]
 (v2) .. controls +(0,0.5) and +(0.5,-0.5) .. (merge);

 \draw[thick, arrowmid]
 (merge) -- (split);

 \draw[thick]
 (split) .. controls +(-0.5,0.5) and +(0,-0.5) .. (v3);
 \draw[thick]
 (split) .. controls +(0.5,0.5) and +(0,-0.5) .. (v4);

 \draw[thick](v3) -- (v5);

 \fill (v1) circle (1.5pt) node[below=2pt] {$v_1$};
 \fill (v2) circle (1.5pt) node[below=2pt] {$v_2$};
 \fill (v3) circle (1.5pt) node[above=2pt] {$v_3$};
 \fill (v4) circle (1.5pt) node[above=2pt] {$v_4$};
 \fill (v5) circle (1.5pt) node[above=2pt] {$v_5$};

\end{tikzpicture}\caption{The second directed hypergraph considered in~\autoref{exa:dirichlet-markovian?}: it consists of two hyperedges -- in fact, one hyperedge and one edge.}
\label{fig:hyper-4-balanced-additional-2}
\end{figure}
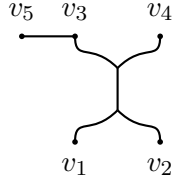
In these cases, the corresponding Laplacians $\mathcal L^{\mV',\mathrm{D}}_\mH$ are
\[
\begin{pmatrix}
0 & 0 & 0 & 0 & 0\\
0 & 1 & -1 & 0 & 0\\
0 & -1 & 2 & 0 & -1\\
0 & 0 & 0 & 0 & 0\\
0 & 0 & -1 & 0 & 1\\
\end{pmatrix},\qquad
\begin{pmatrix}
0 & 0 & 0 & 0 & 0\\
0 & 1 & -1 & 0 & 0\\
0 & -1 & 2 & 0 & 0\\
0 & 0 & 0 & 0 & 0\\
0 & 0 & 0 & 0 & 0\\
\end{pmatrix},
\]
respectively.
\end{enumerate}
\end{exa}

Before discussing the qualitative properties of the semigroup generated by a hypergraph Laplacian with Dirichlet conditions we need to introduce the following notions, which are inspired by analogous notions that for graphs.

\begin{defi}
Let $\mH=(\mV,\mE)$ be a hypergraph and let $\mV'\subset \mV$. 
The \emph{sub-hypergraph induced by} $\mV'$ is the directed hypergraph with vertex set $\mV'$ and edge set $\mE':=\{\me=(\mesour,\metarg)\in \mE:\mesour,\metarg \in \mathcal P(\mV')\}$.

The \emph{ Dirichlet sub-hypergraph} (short: D-sub-hypergraph) \emph{induced by} $\mV'$ is the directed hypergraph with vertex set $\mV'$ and edge set $\mE':=\{\me'=(\mesour\cap \mV',\metarg\cap \mV'):\me=(\mesour,\metarg) \in \mE\}$.
\end{defi}

We stress that given a directed hypergraph $\mH=(\mV,\mE)$ and a vertex subset $\mV'$, neither is the sub-hypergraph induced by $\mV'$ a sub-hypergraph of the D-sub-hypergraph induced by $\mV'$, nor vice versa. 

Observe that the the space of functions supported on the vertices of the D-sub-hypergraph induced by $\mV'$ is not only a subspace, but even an ideal of the space $\ell^2(\mV)$. In this way, it is possible to compare heat flows on different hypergraphs in a canonical way.

\begin{prop}
Let $\mH=(\mV,\mE)$ be a directed hypergraph and let $\mV'\subset \mV$. The semigroup $(\e^{-t\mathcal L^{\mV',\mathrm{D}}_\mH})_{t\ge 0}$ is positive if and only if the D-sub-hypergraph induced by $\mV'$ satisfies \eqref{eq:M-matr-hypercond}.

Let $\mV''$ be a further subset of $ \mV$ such that the D-sub-hypergraph induced by $\mV''$ satisfies \eqref{eq:M-matr-hypercond}. Then the semigroup $(\e^{-t\mathcal L^{\mV',\mathrm{D}}_\mH})_{t\ge 0}$ dominates the semigroup $(\e^{-t\mathcal L^{\mV'',\mathrm{D}}_\mH})_{t\ge 0}$ if and only if $\mV''\subset \mV'$.
\end{prop}

\begin{proof}
The first assertion is an immediate consequence of \autoref{prop:posit}.
The second assertion follows from \cite[Corollary~2.22]{Ouh05}.
\end{proof}

\begin{rem}
Unlike in the case of graphs $\mG$, where the bottom of the spectrum of the part of $\mathcal L^{\mV',\mathrm{D}}_\mG$ in $\ell^2(\mV')$ is strictly positive, in the case of directed hypergraphs the bottom of the spectrum of the Laplacian on a D-sub-hypergraph in $\ell^2(\mV')$ can be 0, as shown by the third set $\mV'$ in \autoref{exa:dirichlet-markovian?}.(\ref{item:dirichmark1}).
\end{rem}

\subsection{Examples}


\begin{enumerate}[(1)]
\item\label{item:posit-1hyper} Consider the hypergraphs in \autoref{exa:prototyp}.(\ref{item:onehyperedge}): by~ \autoref{prop:posit}, $(\e^{-t\mathcal L})_{t\ge 0}$ is positive if and only if $d_+=d_-=1$ (corresponding to the Laplacian on a directed graph on two vertices and one edge); in this case it is also irreducible, since so is the matrix $\mathcal L$; and, by~\autoref{cor:stochequi}, stochastic, since $\mH$ is equipotent. Furthermore, $d_+=d_-=1$ are also necessary conditions for the lowest eigenvalue to be simple and for the corresponding eigenspace to be spanned by a strictly positive vector, hence for eventual irreducibility. Even mere asymptotic positivity requires $d_+=d_-=1$, since otherwise the orthogonal projector onto $\ker\mathcal L$ would not be a positive matrix. Also, let $\#\mV\ge 3$ and observe that $(\e^{t\mathcal L})_{t\ge 0}$ for $d_-=0$ and $d_+=\#\mV$ (or, equivalently, $d_-=\#\mV$ and $d_+=0$) is positive: indeed, it is the modulus semigroup of all semigroups generated by $-\mathcal L$ for $d_-\ge 1$ and $d_+\ge 1$.

Finally,  $(\e^{-t\mathcal L})_{t\ge 0}$ is $\infty$-contractive if and only if $d_- +d_+=2$  (corresponding to either the Laplacian or the signless Laplacian on a directed graph on two vertices and one edge). 

Let now $d_- + d_+\ge 3$. Based on \eqref{eq:1vector-proj}, the norm of the orthogonal projector onto the null space of $\mathcal L$ as an operator on $\ell^\infty(\mV)$ has norm $2\frac{\#\mV-1}{\#\mV}>1$. By \autoref{prop:eventlinftycontr} we conclude that $(\e^{-t\mathcal L})_{t\ge 0}$ is not even asymptotically $\infty$-contractive.

In fact, it can be checked that the semigroup $(\e^{-t\mathcal L})_{t\ge 0}$ is given by
\[
\e^{-t\mathcal L}=
\Id_{\#\mV}+\frac{1}{\#\mV}(\e^{-t\#\mV}-1)\cdot\begin{pmatrix}
\begin{tabular}{c|c}
$J_{d_-, d_-}$ & $- J_{d_-, d_+}$ \\ 
\hline 
$-J_{d_+, d_-}$ & $ J_{d_+, d_+}$ \\ 
\end{tabular} 
\end{pmatrix},\qquad t\ge 0,
\]
whose norm as an operator on $\ell^\infty(\mV)$ is, independently of $d_-,d_+$ but only of $d_-+d_+=\#\mV$,
\[
\begin{split}
\frac{1}{\#\mV}\left|\e^{-t\#\mV}+(\#\mV-1)\right|+\frac{\#\mV-1}{\#\mV}\left|\e^{-t\#\mV}-1\right|
&=
 2\frac{\#\mV-1}{\#\mV}-\frac{\#\mV-2}{\#\mV}\e^{-t\#\mV}.
\end{split}
\]


\item The Laplacian associated with the hypergraphs in \autoref{exa:prototyp}.(\ref{item:hypersignless}) generates a semigroup that, by~\eqref{lem:glu}, is not even asymptotically positive for any $\#\mV\ge 2$, since the orthogonal projector onto $\ker \mathcal L$ is not a positive matrix. In view of \autoref{lem:mug07}, the semigroup is $\infty$-contractive if and only if $\#\mV=2$. However, observe that  the backward heat flow -- i.e., $(\e^{t\mathcal L})_{t\ge 0}$ -- is positive and irreducible (but neither $\infty$-contractive nor stochastic) for any $\#\mV\ge 2$.

\item The hypergraphs in \autoref{exa:prototyp}.(\ref{item:hyperrotat}) have a richer theory. To begin with, observe that $\mathcal L$ is a Z-matrix -- equivalently, $(\e^{-t\mathcal L})_{t\ge 0}$ is positive -- if and only $\#\mV\le 4$. However, $\mathcal L$ is only irreducible in the cases of $\#\mV=2$ or $\#\mV=3$, since for $\#\mV=4$ the matrix $\mathcal L$ is diagonal; and stochastic if and only if $\#\mV=2$. Also, for $\#\mV\ge 4$ the lowest eigenvalue 4 is not simple, hence by~\autoref{lem:glu} $(\e^{-t\mathcal L})_{t\ge 0}$ is not eventually irreducible; and not even asymptotically positive for $\#\mV\ge 5$, since the orthogonal projector onto the eigenspace associated with the lowest eigenvalue 4 is not a positive matrix. By \autoref{lem:mug07}, $(\e^{-t\mathcal L})_{t\ge 0}$ is $\infty$-contractive if and only if $\#\mV\le 5$.
In particular, $(\e^{-t\mathcal L})_{t\ge 0}$ is sub-Markovian if and only if $\#\mV=2,3,4$; and stochastic if and only if $\#\mV=2$.

\item\label{item:plotgraphyhper} Likewise, let $\mH$ be the directed hypergraph with incidence matrix
\[
\mathcal I_\mH:=\begin{pmatrix}
1 & -1 & 0\\
1 & 1 & -1\\
1 & 0 & 1
\end{pmatrix},\qquad\hbox{whence}\qquad 
\mathcal L_\mH:=\begin{pmatrix}
2 & 0 & 1\\
0 & 3 & 0\\
1 & 0 & 2
\end{pmatrix},
\]
which, in view of \autoref{exa:prototyp}, is a special case of the setting considered in \autoref{prop:expstable-ex1}: more precisely, $\mH$ is the union of a path graph $\mG$ on three vertices, with
\[
\mathcal I_\mG:=\begin{pmatrix}
 -1 & 0\\
 1 & -1\\
 0 & 1
\end{pmatrix},\qquad\hbox{whence}\qquad 
\mathcal L_\mG:=\begin{pmatrix}
1 & -1 & 0\\
-1 & 2 & -1\\
0 & -1 & 1
\end{pmatrix},
\]
 and a single hyperedge.
Then $(\e^{-t\mathcal L_\mH})_{t\ge 0}$ not even asymptotically positive, since the lowest eigenvalue 1 is simple but the corresponding eigenspace is spanned by the sign-changing vector $\varphi_1=\frac{1}{\sqrt{2}}\begin{pmatrix}
-1 & 0 & 1
\end{pmatrix}^\top$.
This semigroup is $\infty$-contractive by \autoref{lem:mug07}: this shows that the assumption that all vertices of $\mH$ are co-oriented cannot be dropped in \autoref{lem:hypersignl}.
Also, we know from~\autoref{prop:arorgluunion} that $(\e^{-t\mathcal L_\mH})_{t\ge 0}$ is eventually dominated by $(\e^{-t\mathcal L_\mG})_{t\ge 0}$: we have
\[
\e^{-t\mathcal L_\mH}\begin{pmatrix}
1 \\ 0 \\ 0
\end{pmatrix}=\frac12 \begin{pmatrix}
\e^{-3t}+\e^{-t}\\ 0 \\ \e^{-3t}-\e^{-t}
\end{pmatrix}\qquad\hbox{and}\qquad
\e^{-t\mathcal L_\mG}\begin{pmatrix}
1 \\ 0 \\ 0
\end{pmatrix}=\frac{1}{6} \begin{pmatrix}
3\e^{-t}+\e^{-3t}+2\\ -2\e^{-3t}+2\\ \e^{-3t}-3\e^{-t}+2
\end{pmatrix},\qquad t\ge 0,
\]
as well as
\[
\e^{-t\mathcal L_\mH}\begin{pmatrix}
0 \\1 \\ 0 
\end{pmatrix}=\frac12 \begin{pmatrix}
 0 \\ \e^{-3t} \\ 0
\end{pmatrix}\qquad\hbox{and}\qquad
\e^{-t\mathcal L_\mG}\begin{pmatrix}
0 \\ 1 \\ 0
\end{pmatrix}=\frac{1}{3} \begin{pmatrix}
-\e^{-3t}+1\\ 2\e^{-3t}+1\\ -\e^{-3t}+1
\end{pmatrix},\qquad t\ge 0,
\]
and so we find that the latter semigroup dominates the former only for all $t\ge t_0\approx 1.006$, see \autoref{fig:eventdom}.
\begin{figure}[ht]
\includegraphics[scale=.45]{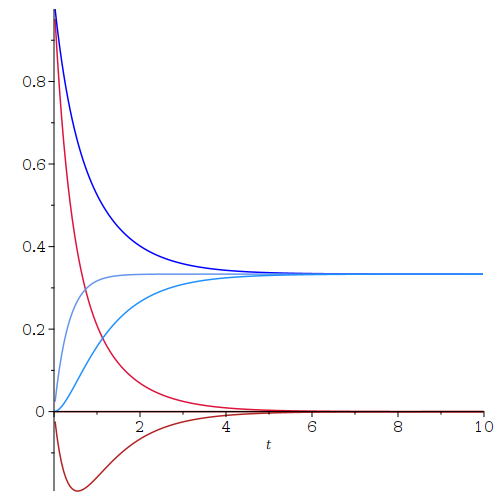} 
\includegraphics[scale=.45]{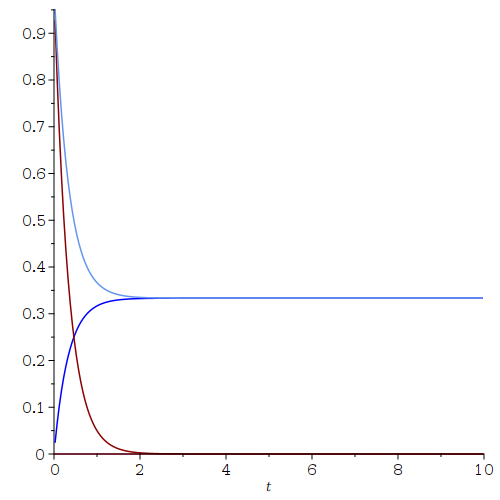} 
\caption{The heat flow driven by $-\mathcal L_\mH$ (red) and $-\mathcal L_\mG$ (blue) in (\ref{item:plotgraphyhper}) for the initial conditions 
$u_0=\begin{pmatrix}
1 & 0 & 0 
\end{pmatrix}^\top$ as well as $u_0=\begin{pmatrix}
0 & 0 & 1
\end{pmatrix}^\top$ (left); and for the initial condition 
$u_0=\begin{pmatrix}
0 & 1 & 0 
\end{pmatrix}^\top$ (right).
}
\label{fig:eventdom}.
\end{figure}

\item 
Consider the directed hypergraph $\mH_1$ with incidence matrix
\[
\mathcal I_1=\begin{pmatrix}
-1 & 0 & 0 \\
1 & 	-1 & 0\\
1 & 0 & -1 \\
0 & 1 & -1
\end{pmatrix}\qquad \hbox{whence}\qquad 
\mathcal L_1=\begin{pmatrix}
1 & -1 & -1 & 0\\
-1 & 2 & 1 & -1\\
-1 & 1 & 2 & -1\\
0 & -1 & -1 & 2
\end{pmatrix}.
\]
Because some of its off-diagonal entries are positive, $\mathcal L_1$ is not a Z-matrix: accordingly, $(\e^{-t\mathcal L_1})_{t\ge 0}$ is not a positive semigroup. However, $0$ is a simple eigenvalue of $\mathcal L_1$ whose corresponding eigenspace is spanned by the normed vector $\varphi_1=\frac{1}{\sqrt{7}}\begin{pmatrix}
2& 1 & 1& 1\end{pmatrix}^\top$; 
hence $(\e^{-t\mathcal L_1})_{t\ge 0}$ is eventually irreducible by \autoref{lem:glu}. Indeed, using Maple 2024 we find
that
 $\e^{-t\mathcal L_1}\ge 0$ only for all $t\ge t_0\approx 1.216$. However, \eqref{eq:ellinftysi} is not satisfied, hence $(\e^{-t\mathcal L_1})_{t\ge 0}$ is not asymptotically $\infty$-contractive.
 
The same holds considering  the directed hypergraph $\mH_2$ with incidence matrix
\[
\mathcal I_2=\begin{pmatrix}
1 & 1 & 1 & 1 \\
1 & 	-1 & 0 & 0\\
-1 & 0 & -1& 0 \\
-1 & 0 & 0 & -1
\end{pmatrix}\qquad \hbox{whence}\qquad 
\mathcal L_2=\begin{pmatrix}
4 & 0 & -2 & -2\\
0 & 2 & -1 & -1\\
-2 & -1 & 2 & 1\\
-2 & -1 & 1 & 2
\end{pmatrix}.
\]
In this case, the null space of $\mathcal L_2$ is spanned by the normed vector
\(\varphi_1=\frac{1}{2}\begin{pmatrix} 1& 1& 1& 1 \end{pmatrix}^\top\):
hence $(\e^{-t\mathcal L_2})_{t\ge 0}$ is eventually irreducible and indeed $\e^{-t\mathcal L_2}\ge 0$ only for all $t\ge t_0\approx 0.655$. Furthermore, $\|\varphi_1\|_1 \|\varphi_1\|_\infty=1$ and, hence, \eqref{eq:ellinftysi} is satisfied: accordingly, by \autoref{cor:eventlinftycontr-matr} $(\e^{-t\mathcal L_2})_{t\ge 0}$ is asymptotically $\infty$-contractive.
Because the directed hypergraph induced by $\mathcal I_2$ is even equipotent, we can apply \autoref{prop:greiner} and conclude that $(\e^{-t\mathcal L_2})_{t\ge 0}$ is  $\infty$-contractive, too.

%
%
%

\item 
Consider the equipotent directed hypergraph $\mH$ with signed incidence matrix
\[
\mathcal I_\mH:=\begin{pmatrix}
-1 \\ -1 \\ 1 \\ 1
\end{pmatrix}\qquad\hbox{whence}\qquad 
\mathcal L_{\mH}=\begin{pmatrix}
1 & 1 & -1 & -1\\
1 & 1 & -1 & -1\\
-1 & -1 & 1 & 1\\
-1 & -1 & 1 & 1
\end{pmatrix}
\]
We already know from (\ref{item:posit-1hyper}) that $(\e^{-t\mathcal L})_{t\ge 0}$ is neither positive nor $\infty$-contractive (and, in fact, not even asymptotically positive or eventually $\infty$-contractive). However, the Markov property of the heat flow can be achieved by suitably modifying the directed hypergraph, as the follow examples show.
\begin{itemize}
\item Connect with one edge the vertices in $\me_{\rm sour}$ and with another edge the vertices in $\me_{\rm targ}$,: then
\[
\mathcal I_{\mH'}=\begin{pmatrix}
-1 & -1 & 0\\
-1 & 1 & 0\\
1 & 0 & -1\\
1 & 0 & 1
\end{pmatrix}\quad \hbox{whence}\quad \mathcal L_{\mH'}=\begin{pmatrix}
2 & 0 & -1 & -1\\
0 & 2 & -1 & -1\\
-1 & -1 & 2 & 0\\
-1 & -1 & 0 & 2
\end{pmatrix}.
\]
\item Add to $\mH$ one hyperedge $\me'$ such that $\me_{\rm sour},\me_{\rm targ}$ both contain one source and one target vertex of $\mH$: then
\[
\mathcal I_{\mH''}=\begin{pmatrix}
-1 & -1\\
-1 & 1\\
1 & -1\\
1 & 1 
\end{pmatrix}\quad \hbox{whence}\quad \mathcal L_{\mH''}=\begin{pmatrix}
2 & 0 & 0 & -2\\
0 & 2 & -2 & 0\\
0 & -2 & 2 & 0\\
-2 & 0 & 0 & 2
\end{pmatrix}.
\]
Observe that neither semigroup $(\e^{-t\mathcal L_{\mH'}})_{t\ge 0},(\e^{-t\mathcal L_{\mH''}})_{t\ge 0}$ is the modulus semigroup of $(\e^{-t\mathcal L_\mH})_{t\ge 0}$.

In either case, the semigroup driving the heat flow is stochastic.
 \end{itemize}
\item In the case where $\mH$ is a general hypergraph consisting of one hyperedge and $\#\mV\ge 3$, positivity of $(\e^{-t\mathcal L})_{t\ge 0}$  can be restored, as in (6), adding a hyperedge $\me=\vv{\{\mv\},\{\mw\}}$ between any two vertices that are co-oriented within $\mH$, thus forming a directed hypergraph $\mH'$.

For instance, consider the  oriented hypergraph with signed incidence matrix
\begin{equation}\label{eq:hyper3a-decor}
\mathcal I_{\mH'}= \begin{pmatrix}
-1 & 0\\
1 & -1\\
1 & 1
\end{pmatrix}\quad\hbox{whence}\quad 
\mathcal L_{\mH'}=
\begin{pmatrix}
1 & -1 & -1\\
-1 & 2 & 0\\
-1 & 0 & 2
\end{pmatrix},
\end{equation}
whose eigenvalues are $0,2,3$ with null space spanned by $\frac{1}{\sqrt{6}}\begin{pmatrix}
2 & 1 & 1
\end{pmatrix}^\top$.
Again by \autoref{prop:posit} and \autoref{lem:mug07}, $(\e^{-t\mathcal L})_{t\ge 0}$ is positive and irreducible but not $\infty$-contractive -- indeed, by \autoref{cor:eventlinftycontr-matr} not even asymptotically $\infty$-contractive, since $J_0$ is a singleton and for the $\|\cdot\|_2$-normed eigenvector $\varphi_1$ corresponding to the eigenvalue 0 one finds
\(
\| \varphi_1\|_1 \|\varphi_1\|_\infty=\frac{4}{3}>1
\),
i.e., \eqref{eq:ellinftysi} is not satisfied.
\end{enumerate}

\section{Dual directed hypergraphs}\label{sec:graphypergra}.

The class of directed graphs is not closed under duality: that is, the dual of a directed graph is not necessarily a directed graph, but generally only a directed hypergraph. It is all the more interesting that the theory of hypergraphs is, in fact, quite rich.

\subsection{Directed hypergraphs dual to graphs}
Let $\mG$ be a directed graph with incidence matrix $\mathcal I_\mG$; in agreement with the general case introduced in \autoref{rem:misc}.(\ref{remitem:josmulrol}) we call the directed hypergraph with incidence matrix $\mathcal I^\top_\mG $ the \textit{dual directed hypergraph} of $\mG$, and we denote it by $\mG^*$. Accordingly, its Laplacian is $\mathcal L_{\mG^*}=\mathcal I^\top_\mG \mathcal I_\mG$.

We can now refine \autoref{lem:gersh0}.

\begin{lemma}
Let $\mG$ be a directed graph. Then all eigenvalues of $\mathcal L_{\mG^*}$ lie in $[0,2 \deg_{\max}(\mV)]$.

\end{lemma}
\begin{proof}
In view of \eqref{eq:firstgershc-graph}, the assertion is an immediate consequence of the isospectrality of $\mathcal L_{\mG^*}=\mathcal I^\top \mathcal I$ and $\mathcal L_{\mG}=\mathcal I \mathcal I^\top$ away from 0.
\end{proof}


\begin{prop}
Let $\mG=(\mV,\mE)$ be a simple directed graph. Then the multiplicity of $0$ as an eigenvalue of the dual Laplacian agrees with the number of independent cycles of $\mG$.

Accordingly,
$(\e^{-t\mathcal L_{\mG^*}})$ is exponentially stable if and only if $\mG$ is a forest.
\end{prop}

%


\begin{proof}
By \eqref{eq:ranknul} we find
\[
\begin{split}
\dim\ker \mathcal L_{\mG^*}
&=\#\mE-\#\mV+c,
\end{split}
\]
since it is well-known that $\dim\ker \mathcal I^\top=\dim\ker \mathcal L_\mG$ agrees with the number $c$ of connected components of $\mG$. Now the assertions follows, as $\#\mE-\#\mV+c$ is the cyclomatic number of $\mG$ (regardless of the chosen orientation of the edges).

In particular, $\mG$ is a forest if and only if $0$ is not an eigenvalue of $\mathcal L_{\mG^*}$: in this case, the lowest eigenvalue must be strictly positive, and the claim follows.
\end{proof}

\begin{prop}\label{prop:chatg-posit}
Let $\mG=(\mV,\mE)$ be a simple directed graph.
Then $(\e^{-t\mathcal L_{\mG^*}})$ is positive if and only if $\deg_{\max}(\mV)\le 2$. 
\end{prop}

\begin{proof}
%
%
%
%
%
%
%

We have $(\mathcal I^\top \mathcal I)_{\me\mf}=\iota_{\me\cdot}\cdot \iota_{\mf\cdot}$. For distinct edges $\me\ne \mf$, this inner product is $\pm1$ if the edges $\me,\mf$ meet at one vertex (the sign depends on the orientations at the common vertex), and is $0$ otherwise. The diagonal entries are $ (\mathcal I^\top \mathcal I)_{\me\me}=2$ for a simple directed graph, since each edge has two non-zero entries $+1,-1$.

Let us show that an orientation of $\mG$ that turns $\mathcal I^\top\mathcal I$ into a Z-matrix can be found if and only if $\deg_{\max}(\mV)\le 2$.

``$\Rightarrow$''
Let $\mv\in \mV$ with $\deg(\mv)\ge3$ and let $\me_1,\me_2,\me_3$ be three distinct edges incident to $\mv$. In the rows $\iota_{\cdot\me_i }$ of $\mathcal I$ the only overlapping coordinate among these rows (for the three distinct edges) is the coordinate corresponding to vertex $\mv$. At that coordinate each row has either $+1$ or $-1$. Thus the inner product $\iota_{\cdot\me_i}\cdot \iota_{\cdot\me_j}$ for $i\ne j$ equals the product of these two signs (since two edges can be simultaneously incident in at most one vertex). Accordingly, at least two of these three products must agree, say
$\iota_{\cdot\me_i}\cdot \iota_{\cdot\me_j}=+1$. We have hence proved that there is no orientation making all off-diagonal entries of $\mathcal I^\top\mathcal I$ nonpositive. 

``$\Leftarrow$'' If $\deg_{\max}(\mV)\le2$, then $\mG$ is a disjoint union of isolated vertices, paths and cycles. It is hence possible to orient each connected component coherently in such a way that any vertex $\mv$ of degree $2$ has exactly one incoming edge (say, $\me$) and one outgoing edge (say, $\mf$).

Therefore, the two rows of $\mathcal I$ corresponding to $\me,\mf$ have opposite signs at $\mv$, so their contribution to the inner product is $-1$. For two edges that do not meet (and in particular for products corresponding to vertices of degree 0 or 1) the inner product is $0$. Thus for every pair of distinct edges $e\ne f\in \mE$ we have $(\mathcal I^\top \mathcal I)_{\me\mf}\in\{0,-1\}$, i.e. $\le0$. Hence $\mathcal I^\top \mathcal I$ has nonpositive off-diagonals, so $-\mathcal L_{\mG^*}$ generates a positive semigroup.
\end{proof}

%
%
%
%
%
%

\begin{cor}
Let $\mG=(\mV,\mE)$ be a simple directed graph. If $\deg_{\max}(\mV)\le2$, then $(\e^{-t\mathcal L^*})_{t\ge 0}$ is sub-Markovian.

Furthermore, 
$(\e^{-t\mathcal L_{\mG^*}})$ is stochastic if and only if $\deg\equiv 2$.
\end{cor}

\begin{proof}
Let $\deg_{\max}(\mV)\le 2$, whence by \autoref{prop:chatg-posit} $(\e^{-t\mathcal L^*})_{t\ge 0}$ is positive.

If $\deg_{\mG}\equiv 2$, and hence each connected component of $\mG$ is a cycle, then the incidence matrix $\mathcal I$ is a permutation of its transpose $\mathcal I^\top$ and, accordingly, $\mathcal L_{\mG^*}$ coincides, up to a permutation, with $\mathcal L_{\mG}$, hence $(\e^{-t\mathcal L^*})_{t\ge 0}$ is stochastic, too, and in particular sub-Markovian.

If $\deg\not\equiv 2$, hence the minimal degree of $\mG$ is 1, then a connected component of $\mG$ consists of a path, then one can prove by induction that $\mathcal L^*$ is a block matrix all of whose blocks are tri-diagonal matrices whose diagonal entries are all 2 and whose upper and lower diagonal entries are all $-1$. Accordingly, $\mathbf{1}\not\in\ker (\mathcal L^*)$: a contradiction to the stochastic property of $(\e^{-t\mathcal L^*})_{t\ge 0}$. However $\mathcal L^*\mathbf{1}\le 0$, which implies the sub-Markovian property.
\end{proof}

%
%

\begin{exa}
\begin{enumerate}
\item Consider the directed path graph $\mG$ on four vertices. The corresponding Laplacian and dual Laplacians are
\begin{equation}\label{eq:laplduallap->0}
\mathcal L=
\begin{pmatrix}
1 & -1 & 0 & 0\\
-1 & 2 & -1 & 0\\
0 & -1 & 2 & -1 \\
-1 & 0 & -1 & 2 \\
\end{pmatrix}
\qquad \hbox{and}\qquad
\mathcal L^*=\begin{pmatrix}
2 & -1 & 0\\
-1 & 2 & -1 \\
0 & -1 & 2 \\
\end{pmatrix}
\end{equation}
The eigenvalues of the latter are $2-\sqrt{2},2,2+\sqrt{2}$. The semigroup $(\e^{-t\mathcal L^*})_{t\ge 0}$ is sub-Markovian, but neither Markovian nor stochastic; indeed, the invariant measure is given by the normed vector \(\varphi_1=\frac{1}{2}\begin{pmatrix} 1& \sqrt{2}& 1\end{pmatrix}^\top\)
that spans the lowest eigenvalue
that is, the rescaled semigroup satisfies
\[
\e^{-t(\mathcal L^*-(2-\sqrt{2}))}\begin{pmatrix}
1\\ \sqrt{2}\\ 1
\end{pmatrix}=\begin{pmatrix}
1\\ \sqrt{2}\\ 1
\end{pmatrix}\qquad\hbox{for all }t\ge 0.
\]
(If we consider the path on three vertices, instead, then the behaviour of $\mathcal L^*$ is similar, but in this case the semigroup $(\e^{-\mathcal L^*})_{t\ge 0}$ is even stochastic.)

\item This property does not extend to the dual Laplacian of \emph{any} directed graph. For instance, if $\mG$ is a star on three edges, then depending on the edges' orientation the dual Laplacian $\mathcal L^*$ is
\[
\begin{pmatrix}
2 & 1 & 1\\
1 & 2 & 1\\
1 & 1 & 2
\end{pmatrix},\quad
\begin{pmatrix}
2 & -1 & -1\\
-1 & 2 & 1\\
-1 & 1 & 2
\end{pmatrix},\quad
\begin{pmatrix}
2 & -1 & 1\\
-1 & 2 & -1\\
1 & -1 & 2
\end{pmatrix},\quad\hbox{or}\quad
\begin{pmatrix}
2 & 1 & -1\\
1 & 2 & -1\\
-1 & -1 & 2
\end{pmatrix}:
\]
none of these matrices generates a sub-Markovian semigroup.
Also, the eigenvalues of all these matrices are $1^{(2)},4$: the lowest eigenspace is spanned by
\[
\frac{1}{\sqrt{2}}\begin{pmatrix}
-1\\ 1\\ 0
\end{pmatrix},\qquad 
\frac{1}{\sqrt{6}}\begin{pmatrix}
1\\ 1\\ -2
\end{pmatrix}
\]
 in the former case, and by
\[
\frac{1}{\sqrt{2}}\begin{pmatrix}
1\\ 1\\ 0
\end{pmatrix},\qquad 
\frac{1}{\sqrt{6}}\begin{pmatrix}
1\\ -1\\ -2
\end{pmatrix}
\]
 in the last three cases: 
accordingly, the orthogonal projector onto the eigenspace corresponding to the eigenvalue 1 is 
\[
\frac{1}{3}\begin{pmatrix}
2 & -1 & -1\\
-1 & 2 & -1\\
-1 & -1 & 2
\end{pmatrix}
\qquad\hbox{and}\qquad
\frac{1}{3}\begin{pmatrix}
2 & 1 & 1\\
1 & 2 & -1\\
1 & -1 & 2
\end{pmatrix}
\] 
(or row and column permutations of the latter), respectively. 
Neither of these matrices is either positive or an $\infty$-contraction: accordingly, by \autoref{lem:glu} and \autoref{prop:eventlinftycontr} we deduce that the corresponding heat flows are neither asymptotically positive nor asymptotically $\infty$-contractive.
\end{enumerate}
\end{exa}

\subsection{Abstract simplicial complexes as directed hypergraphs}\label{sec:abstrsimpcomphyper}

It is well-known known that any (oriented) simplicial complex $\mathsf{K}$ can be regarded as an (oriented) hypergraph $\mH$. What about the associated Laplacians? 


We follow the pioneering paper \cite{HorJos13} and consider an \emph{abstract simplicial complex} $\mK$ on a finite set of vertices $\mV$, i.e., a collection of subsets of $\mV$ that is closed under inclusion. An element of $\mK$ of cardinality $i+1$ is called an \emph{$i$-face}: the collection of all $i$-faces in $\mK$ is denoted by $S_i(\mK)$, and $\bigcup_{0\le j\le i}S_i(\mK)$ is called the \emph{$i$-skeleton} of $\mK$.

Let $\mK$ be a simplicial complex of dimension $N$, i.e., $S_N(\mK)$ is non-empty but $S_{N+1}(\mK)$ is, and for $i=0,\ldots,N-1$ denote by 
$\delta_i: \K^{S_i(\mK)}\to \K^{S_{i+1}(\mK)}$ its \emph{coboundary operator}.
We can turn $S_i(\mK)$ into a hypergraph $\mH_i$ with vertex set $\mV_{\mH_i}:=S_{i-1}(\mK)\cup S_{i+1}(\mK)$ and hyperedge set $\mE_{\mH_i}:=S_{i}(\mK)$ by introducing the incidence matrix
\[
\mathcal I_{\mH_i}:=\begin{pmatrix}
\delta_{i-1}^* \\ \delta_i
\end{pmatrix}
\]
(with the convention that $\delta_{-1}=0$).
Accordingly, $\mathcal I_{\mH^*_i}=\mathcal I_{\mH_i}^\top=\begin{pmatrix}
\delta_{i-1} & \delta_i^*
\end{pmatrix}$ (i.e., the dual hypergraph $\mH^*_i$ is the union of the directed hypergraphs with incidence matrices $\delta_{i-1}$ and $\delta_i^*$, respectively) and we conclude that 
\[
\mathcal L_{\mH^*_i}=\mathcal I_{\mH_i}^\top \mathcal I_{\mH_i}=\delta_{i-1} \delta_{i-1}^* + \delta_i \delta_i^*:
\]
in other words,  the $i$-dimensional Hodge Laplacian $\mathcal L^{(i)}_{\mK}$ of $\mK$ is precisely the dual Laplacian $\mathcal L_{\mH^*_i}$. For $i=0$, this is nothing but the usual Laplacian of (an arbitrary orientation of) the $0$-skeleton of $\mK$ (regarded as a graph). For $i\ge 1$, the heat flow driven by $-\mathcal L^{(i)}_\mK$ has been studied in~\cite{TorBia20}, where the focus was on the topological interpretation of the stationary solutions.

\begin{exa}\label{exa:simplcompl}
\begin{enumerate}[(1)]
\item If $\mK$ is the 2-simplex consisting of the three vertices $\mv_1,\mv_2,\mv_3\in S_0(\mK)$, of the oriented edges $(\mv_1,\mv_2),(\mv_3,\mv_2),(\mv_3,\mv_1)\in S_1(\mK)$, and of the face $\sigma\in S_2(\mK)$, then the co-boundary operators are given by
\[
\delta_0=\begin{pmatrix}
0 & -1 & 1\\
-1 & 0 & 1\\
-1 & 1 & 0
\end{pmatrix}\qquad\hbox{and}\qquad
\delta_1=\begin{pmatrix}
1 & -1 & 1
\end{pmatrix}.
\]

There is no orientation of $\mK$ that turns $\delta_1$ into the all-1 vector and, thus, we cannot immediately apply \autoref{lem:spectralgeomhyper} to this setting; though, we can invoke \autoref{rem:notation-after-josmul} to swap the roles of the vertices in the second hyperedge of $\mH_i$ and hence the sign of all entries in the second column of $\mathcal I_{\mH_i}$: the Hodge Laplacian is invariant under this swapping and we conclude that $(\e^{-t\mathcal L^{(1)}_\mK})_{t\ge 0}$ is positive and uniformly exponentially stable by \autoref{prop:arorgluunion} and \autoref{cor:frogexponential}, respectively.
Indeed,  a direct computation shows that the Hodge Laplacian acting on the edge space is
\[
\mathcal L^{(1)}_\mK=\begin{pmatrix}
3 & 0 & 0\\
0 & 3 & 0\\
0 & 0 & 3
\end{pmatrix}.
\]

\item\label{exaitem:torbiaest} Elaborating on a specific example from~\cite{TorBia20} (a simplicial complex consisting of three triangles sharing two edges), these ideas  were further discussed in \cite[Section~4.2]{MirEstEst23}: therein, an example of non-positive heat flow was worriedly observed on the simplicial complex with coboundary operators
\[
\delta_0:=
\begin{pmatrix}
-1& -1& 0& 0& 0& 0& 0\\
1& 0& -1& 0& -1& -1& 0\\
0& 1& 1& -1& 0& 0& 0\\
0& 0& 0& 1& 1& 0& -1\\
0& 0& 0& 0& 0& 1& 1
\end{pmatrix}\qquad \hbox{and}\qquad 
\delta_1:=
\begin{pmatrix}
1& 0& 0\\
-1& 0& 0\\
1& 1& 0\\
0& 1& 0\\
0& -1& 1\\
0& 0& -1\\
0& 0& 1
\end{pmatrix},
\]
both at the level of edges and at level of faces.
Indeed, the relevant Hodge Laplacian turns out to be 
\[
\mathcal L^{(1)}_\mK=\begin{pmatrix}
3 & 0 & 0 & 0 & -1 & -1 & 0\\
0 & 3 & 0 & -1 & 0 & 0 & 0\\
0 & 0 & 4 & 0 & 0 & 1 & 0\\
0 & -1 & 0 & 3 & 0 & 0 & -1\\
-1 & 0 & 0 & 0 & 4 & 0 & 0 \\
-1 & 0 & 1 & 0 & 0 & 3 & 0\\
0 & 0 & 0 & -1 & 0 & 0 &3
\end{pmatrix}\qquad \hbox{and}\qquad 
\mathcal L^{(2)}_\mK=
\begin{pmatrix}
3 & 1 & 0\\
1 & 3 & -1\\
0 & -1 & 3
\end{pmatrix}
\]
and one can check 
that for both matrices the orthogonal projector onto the eigenspace associated with their lowest eigenvalue ($3-\sqrt{2}$ for both $\mathcal L^{(1)}_\mK$ and $\mathcal L^{(2)}_\mK$) is not a positive matrix, hence neither $(\e^{-t\mathcal L^{(1)}_\mK})_{t\ge 0}$ nor $(\e^{-t\mathcal L^{(2)}_\mK})_{t\ge 0}$ are even asymptotically positive. Also, neither of them is asymptotically $\infty$-contractive, as \autoref{fig:simplicial-failure} shows. 
\begin{figure}[ht]
\includegraphics[scale=.45]{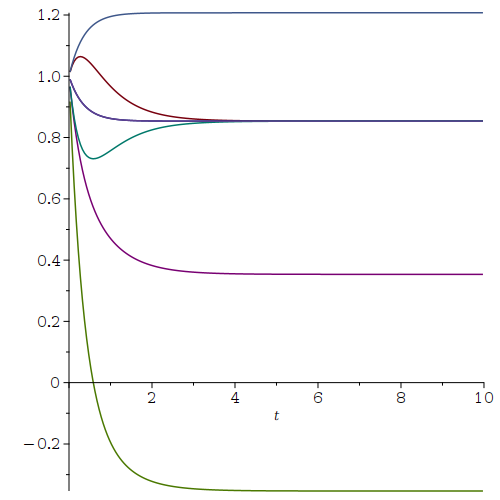} 
\includegraphics[scale=.45]{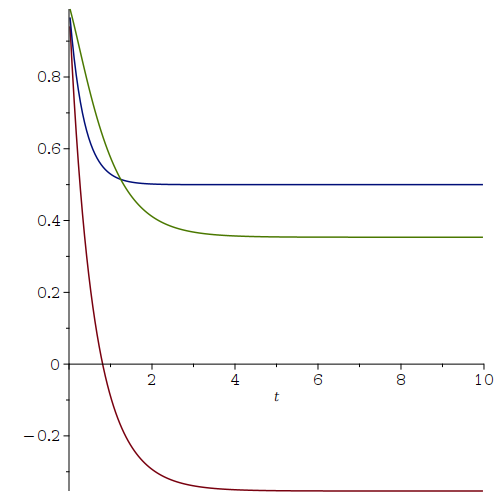} 
\caption{A plot of $\e^{-t\left(\mathcal L^{(1)}_\mK-(3-\sqrt{2})\right)}\mathbf{1}_7$ (left) and $\e^{-t\left(\mathcal L^{(2)}_\mK-(3-\sqrt{2})\right)}\mathbf{1}_3$ (right)  in \autoref{exa:simplcompl}.(\ref{exaitem:torbiaest}).}\label{fig:simplicial-failure}
\end{figure}
Actually, one can check that the $\infty$-norm of
$\left(\e^{-t\left(\mathcal L^{(2)}_\mK-(3-\sqrt{2})\right)}\right)_{t\ge 0}$ is given by
\[
\left\|\e^{-t\left(\mathcal L^{(2)}_\mK-(3-\sqrt{2})\right)}\begin{pmatrix}
0 \\ 1 \\ 0
\end{pmatrix}\right\|_1=
\frac{1}{2} 
\left( \sqrt{2}\left|\e^{-2\sqrt{2}t} - 1\right| +\left|\e^{-2\sqrt{2}t} + 1\right|\right) \stackrel{t\to\infty}{\to} \frac{1+\sqrt{2}}2
\]
Interestingly, this is also the limit value of the $\infty$-norm of
$\left(\e^{-t\left(\mathcal L^{(2)}_\mK-(3-\sqrt{2})\right)}\right)_{t\ge 0}$ as $t\to \infty$.

\end{enumerate}
\end{exa}


\begin{rem}
While the Hodge Laplacians of a simplicial complex can be represented as the Laplacian of suitable directed hypergraphs, the latter cannot generally  be represented as graph Laplacians -- not even  in the \textit{generalised}  sense of Colin de Verdière \cite{Col90}. Indeed,
\begin{itemize}
\item \autoref{exa:prototyp}.(\ref{item:hyperrotat})  shows that the $\mv-\mw$-entry of the Laplacian may vanish even though $\mv,\mw$ belong to a common hyperedge;
\item \autoref{exa:prototyp}.(\ref{item:onehyperedge})--(\ref{item:hypersignless}) shows that 0 may be an eigenvalue of higher multiplicity, see Section~\ref{sec:spectheor};
\item there exist non-zero matrices $X$ that fully annihilate the Laplacian: for example, if we consider
 \[\mathcal I_3=\begin{pmatrix}
1 \\ 0 \\ 0
\end{pmatrix}\quad\hbox{whence}\quad 
\mathcal L_3=\begin{pmatrix}
1 & 0 & 0\\
0 & 0 & 0\\
0 & 0 & 0
\end{pmatrix},
\]
then $\mathcal L X=0$ for $X:=\begin{pmatrix}
0 & 0 & 0\\
0 & 0 & 1\\
0 & 1 & 0
\end{pmatrix}$.
\end{itemize}
\end{rem}

\section{An illustrative example: the Fano plane}\label{sec:fano}

The Fano plane in \autoref{fig:fanoplane} is a classical and, in view of its symmetries, especially interesting example of a hypergraph: the points and the lines drawing represent the hypergraph's vertices and hyperedges, respectively, and a vertex is an element of a hyperedge if the corresponding point lies on the corresponding line.
\begin{figure}[ht]
\begin{tikzpicture}[scale=3, every node/.style={circle,draw,inner sep=1.2pt,fill=white,font=\small}]
\coordinate (A) at (90:1);
\coordinate (B) at (210:1);
\coordinate (C) at (330:1);

\coordinate (Mab) at ($ (A)!0.5!(B) $);
\coordinate (Mbc) at ($ (B)!0.5!(C) $);
\coordinate (Mca) at ($ (C)!0.5!(A) $);

\coordinate (O) at (0,0);

\draw[thick] (A) -- (B);
\draw[thick] (B) -- (C);
\draw[thick] (C) -- (A);

\draw[thick] (A) -- (Mbc);
\draw[thick] (B) -- (Mca);
\draw[thick] (C) -- (Mab);

\draw[thick] (O) circle[radius=0.5]; 

\node[fill=black,label=above:$\mv_1$] (P1) at (A) {};
\node[fill=black,label=left:$\mv_2$] (P2) at (B) {};
\node[fill=black,label=right:$\mv_4$] (P3) at (C) {};

\node[fill=black,label=left:$\mv_3$] (P4) at (Mab) {};
\node[fill=black,label=below:$\mv_6$] (P5) at (Mbc) {};
\node[fill=black,label=right:$\mv_5$] (P6) at (Mca) {};

\node[fill=black,label=left:$\mv_7$] (P7) at (O) {};

\end{tikzpicture}
\caption{A drawing of the Fano plane}\label{fig:fanoplane}
\end{figure}
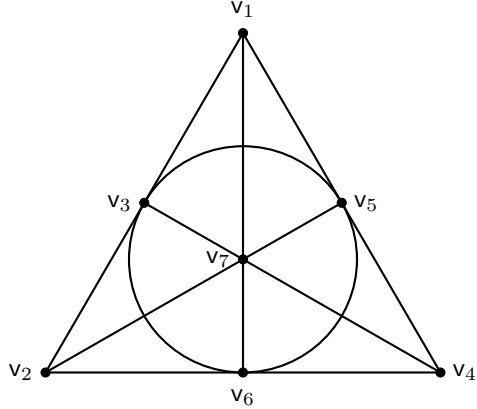
Its incidence matrix and Laplacian are
\begin{equation}\label{eq:lapl-fano-undir}
\mathcal I_F:=
\begin{pmatrix}
1 & 1 & 1 & 0 & 0 & 0 & 0\\
1 & 0 & 0 & 1 & 1 & 0 & 0\\
1 & 0 & 0 & 0 & 0 & 1 & 1\\
0 & 1 & 0 & 1 & 0 & 1 & 0\\
0 & 1 & 0 & 0 & 1 & 0 & 1\\
0 & 0 & 1 & 1 & 0 & 0 & 1\\
0 & 0 & 1 & 0 & 1 & 1 & 0
\end{pmatrix}
\quad\hbox{and}\quad
\mathcal L_F:=
\begin{pmatrix}
3 & 1 & 1 & 1 & 1 & 1 & 1\\
1 & 3 & 1 & 1 & 1 & 1 & 1\\
1 & 1 & 3 & 1 & 1 & 1 & 1\\
1 & 1 & 1 & 3 & 1 & 1 & 1\\
1 & 1 & 1 & 1 & 3 & 1 & 1\\
1 & 1 & 1 & 1 & 1 & 3 & 1\\
1 & 1 & 1 & 1 & 1 & 1 & 3\\
\end{pmatrix}.
\end{equation}
Since $\mathcal L_F$ is neither diagonally dominant, nor a Z-matrix, the semigroup $(\e^{-t\mathcal L})_{t\ge 0}$ is neither $\infty$-contractive, nor positive. What about other orientations of the same hypergraph?

There are $2^{21}$ ways of turning the Fano plane into a directed hypergraph, but -- in view of the invariance property mentioned in \autoref{rem:notation-after-josmul} -- only $2^{14}=16.384$ different Laplacians (including the one in~\eqref{eq:lapl-fano-undir}), and only 112 up to permutations of lines and columns. 

They all share a few properties, including their trace, see \autoref{lem:gersh0}: indeed, the following holds.

\begin{prop}\label{prop:fano-non-posit}
Let $\mH$ be any directed realisation of the Fano plane.
Then, the corresponding semigroups $(\e^{-t\mathcal L})_{t\ge 0}$ is neither positive, nor $\infty$-contractive.
 \end{prop}

\begin{proof} 
To begin with, observe that all diagonal entries of $\mathcal L=\mathcal I\mathcal I^\top$ are 3 (each point lies on 3 lines) and all off-diagonal entries must attain the value $1$ or $-1$ (two distinct points meet in exactly one line, so their row inner product is $\pm1$ determined by the sign choices on that line).

Let now assume the semigroup $(\e^{-t\mathcal L})_{t\ge 0}$ to be positive. Hence, $\mathcal L$ is a Z-matrix, i.e., every off-diagonal would is $-1$. In particular, for any line/hyperedge $\me$ containing three points/vertices $\mv,\mw,\mz$ there holds
\[
{\mathcal L}_{\mv\mw}={\mathcal L}_{\mv\mz}={\mathcal L}_{\mw\mz}=-1.
\]
However, we see that
\[
{\mathcal L}_{\mv\mw}=\iota_{\mv\me} \iota_{\mw\me},\quad {\mathcal L}_{\mv\mz}=\iota_{\mv\me} \iota_{\mz\me},\quad {\mathcal L}_{\mw\mz}=\iota_{\mw\me} \iota_{\mz\me},
\]
so multiplying the first two equalities gives
\[
1=(-1)(-1)={\mathcal L}_{\mv\mw}{\mathcal L}_{\mv\mz}=(\iota_{\mv\me} \iota_{\mw\me})(\iota_{\mv\me} \iota_{\mz\me})=\iota_{\mw\me} \iota_{\mz\me}={\mathcal L}_{\mw\mz}=-1,
\]
a contradiction.

On the other hand, \eqref{eq:diag-dom} is not satisfied -- hence, by \autoref{lem:mug07}, $(\e^{-t\mathcal L})_{t\ge 0}$ is not $\infty$-contractive -- for any realisation.
 \end{proof}
 
However, the directed realisations of the Fano plane may certainly differ in other regards.
In particular, only some of these Laplacian realisations have a simple lowest eigenvalue whose corresponding eigenspace is spanned by a strictly positive vector: hence, by \autoref{lem:glu}, some of the corresponding semigroups $(\e^{-t\mathcal L})_{t\ge 0}$ are eventually irreducible, and some are not. Also, there are realisations, like that induced by
\begin{equation*}\label{eq:lapl-fano-dir-simplebut}
\mathcal I_1:=
\begin{pmatrix}
1 & 1 & 1 & 0 & 0 & 0 & 0\\
1 & 0 & 0 & -1 & 1 & 0 & 0\\
1 & 0 & 0 & 0 & 0 & -1 & 1\\
0 & 1 & 0 & 1 & 0 & 1 & 0\\
0 & -1 & 0 & 0 & -1 & 0 & 1\\
0 & 0 & -1 & -1 & 0 & 0 & 1\\
0 & 0 & -1 & 0 & 1 & 1 & 0
\end{pmatrix}
\quad\hbox{whence}\quad
\mathcal L_1:=
\begin{pmatrix}
3 & 1 & 1 & 1 & -1 & -1 & -1\\
1 & 3 & 1 & -1 & -1 & 1 & 1\\
1 & 1 & 3 & -1 & 1 & 1 & -1\\
1 & -1 & -1 & 3 & -1 & -1 & 1\\
-1 & -1 & 1 & -1 & 3 & 1 & -1\\
-1 & 1 & 1 & -1 & 1 & 3 & 1\\
-1 & 1 & -1 & 1 & -1 & 1 & 3\\
\end{pmatrix},
\end{equation*}
such that 0 is a simple eigenvalue, but the corresponding eigenvector -- $e_1=(1,-1,0,-1,0,0,1)^\top$ in this case -- changes sign, thus by \autoref{lem:glu} preventing $(\e^{-t\mathcal L_1})_{t\ge 0}$ from being even asymptotically positive. Also, $(\e^{-t\mathcal L_1})_{t\ge 0}$ is asymptotically $\infty$-contractive by \autoref{cor:eventlinftycontr-matr}.

Consider the realisations given by the incidence matrices
\[
\mathcal I_2^\pm:=
\begin{pmatrix}
\pm1& \pm 1& -1& 0& 0& 0& 0\\
1& 0& 0& -1& -1& 0& 0\\
1& 0& 0& 0& 0& -1& -1\\
0& 1& 0& 1& 0& 1& 0\\
0& 1& 0& 0& 1& 0& 1\\
0& 0& 1& 1& 0& 0& 1\\
0& 0& 1& 0& 1& 1& 0
\end{pmatrix},\qquad\hbox{whence}\qquad
\mathcal L_2^\pm=
\begin{pmatrix}
3 & \pm 1& \pm 1 & \pm 1& \pm 1& -1& -1\\
\pm 1 & 3 & 1 & -1& -1& -1& -1\\
\pm 1 & 1 & 3 & -1& -1& -1& -1\\
\pm 1 & -1 & -1 & 3 & 1& 1& 1\\
\pm 1 & -1 & -1 & 1 & 3 & 1& 1\\
-1 & -1 & -1 & 1 & 1& 3& 1\\
-1 & -1 & -1 & 1 & 1& 1& 3\\
\end{pmatrix}.
\]
Using Maple 2024 we find that the lowest eigenvalue is $\approx 0.10345$ (which is the lowest root of the polynomial $x^3-13x^2+40x-4$) and that it is simple; that the corresponding eigenspace is spanned by the vector
\[
e_1^-\approx \begin{pmatrix}
2.53982\\
1.67836\\
1.67836\\
1\\
1\\
1\\
1
\end{pmatrix}
\qquad\hbox{or}\qquad
e_1^+\approx \begin{pmatrix}
2.53982\\
-1\\
-1\\
-1.67836\\
-1.67836\\
1\\
1
\end{pmatrix}
\]
respectively.
By  \autoref{rem:inftyposobv}.(\ref{item:inftyposobv-equiv}), both $(\e^{-t\mathcal L_2^\pm})_{t\ge 0}$
are eventually $\infty$-contractive; and by \autoref{cor:eventlinftycontr-matr} neither of them is asymptotically $\infty$-contractive.
 However, 
\autoref{lem:glu}  implies that $(\e^{-t\mathcal L_2^-})_{t\ge 0}$ is eventually irreducible, whereas $(\e^{-t\mathcal L_2^+})_{t\ge 0}$ is not even asymptotically positive.
Indeed, again by Maple 2024 we find that both $\e^{-t\mathcal L_2^\pm}$ are $\infty$-contractive for all $t\ge t_0\approx 4.316$; and that
$\e^{-t\mathcal L_2^-}$ is positive for all $t\ge t_0'\approx 0,727$.

\section{Conclusions and outlook}

Hypergraph Laplacians, viewed through the lens of semigroup theory, occupy an intermediate position between classical graph Laplacians -- quintessentially Markovian -- and general symmetric dissipative operators, which need not preserve order at all. The results established here make this observation precise: the quadratic form associated with the Laplacian on directed hypergraphs is always accretive, but positivity and $\infty$-contractivity -- or their relaxations -- hold only under strong combinatorial restrictions. In most nontrivial cases, the generated semigroup -- while self-adjoint -- is non-Markovian, thus providing a minimal linear model for higher-order, truly nonlocal diffusion.

Our key findings include:
\begin{itemize}
\item 
\underline{Failure of sub-Markovian Properties:} The central finding was that, unlike on graphs, the heat flow on general directed hypergraphs frequently lacks positivity or stochasticity (or both), demonstrating that the Markov property is uncommon in this higher-order setting. This is akin to the behaviour already observed for Hodge Laplacians on simplicial complexes, which we have in fact shown to be special cases of Laplacians on directed hypergraphs.

\item 
\underline{Eventual recovery:} We identified precise combinatorial conditions under which these anomalous dynamics disappear, showing that the heat flow may recover positivity and stochasticity asymptotically after a finite transient period.  It is intriguing that the length of this transient cannot be estimated by  currently available methods.
 
 \item
\underline{Dissensus models:} Dynamical systems driven by graph Laplacians are popular in the complex networks literature since they can be interpreted as simple examples of consensus models; \autoref{fig:difference} suggests dynamical systems driven by directed hypergraph Laplacians may be simple examples of phenomena where the difference between two actors with same initial opinion can be amplified.

 \item \underline{Structural Dependence:} Our analysis  showed that fundamental properties like positivity and consensus equilibrium are highly dependent on the hypergraph's structure. Specific structures, such as duals of oriented graphs and directed realisations of the Fano plane, were used to illustrate these phenomena.
\end{itemize}

The results presented here pave the way for several natural and important directions for future research.
\begin{itemize}
\item Our analysis focused on the linear heat flow. Given the relevance of directed hypergraphs in chemical and cellular networks, extending the analysis to nonlinear heat flows would be a crucial next step, elaborating on the findings in~\cite{Faz23,FazTenLuk24} and particularly investigating if the observed dissensus equilibrium (as seen in \autoref{exa:prototyp}) persists or transforms under nonlinear perturbations. 


\item The physical interpretation of the heat flow on directed hypergraphs and its difference from the heat flow on the hypergraph's 2-section should be better understood, especially with respect to the issue of eventual Markovian dynamics: this should be applied to realistic models in non-equilibrium thermodynamics and molecular or neural networks  \cite{AktAkb21,Dak25} and may provide criteria for predicting when complex higher-order interactions eventually lead to stable, well-behaved diffusive processes over long time scales.

\item
Developing a stronger, combinatorial notion of connectedness for directed hypergraphs that is provably equivalent to the eventual irreducibility of the semigroup is a significant  area for future work. 
\end{itemize}

\bibliographystyle{plain}
\bibliography{../../referenzen/literatur}

\end{document}